\documentclass[11pt]{amsart}

\usepackage[utf8]{inputenc}
\usepackage{verbatim, amssymb,hyperref,float,accents,color,soul,bbm,graphicx,amsmath,tikz,adjustbox,enumerate}
\usepackage[sort,capitalise]{cleveref} 

\crefname{figure}{Figure}{Figures}
\crefname{subsection}{Subsection}{Subsections}
\crefname{enumi}{item}{items}
\crefname{equation}{}{}

\usetikzlibrary{matrix,chains,positioning,decorations.pathreplacing,arrows,shapes,math,arrows.meta}
\tikzset{font={\fontsize{9pt}{12} \selectfont}}

\newcommand{\red}{\color[rgb]{1,0,0}}
\newcommand{\blue}{\color[rgb]{0,0,1}}

\newcommand{\fd}{\mathfrak{d}}
\renewcommand{\d}{\mathrm{d}}

\newcommand{\eps}{\varepsilon}

\newcommand{\cD}{\mathcal D}
\newcommand{\cE}{\mathcal E}

\newcommand{\cH}{\mathcal H}

\newcommand{\cR}{\mathcal R}

\newcommand{\cA}{\mathcal A}
\newcommand{\cO}{\mathcal O}

\newcommand{\cL}{\mathcal L}

\newcommand{\supp}{\operatorname{supp}}

\newcommand{\dd}{\mathrm{d}}
\newcommand{\sfrac}[2]{\mbox{$\frac{#1}{#2}$}}

\newcommand{\IW}{\mathbb W}

\newcommand{\cW}{\mathcal W}

\newcommand{\vecn}{\mathfrak{n}}

\newcommand{\IC}{\mathbb C}
\newcommand{\ID}{\mathbb D}
\newcommand{\II}{\mathbb I}

\renewcommand{\IJ}{\mathbb J}

\newcommand{\1}{1\hspace{-0.098cm}\mathrm{l}}

\newcommand{\N}{{\mathbb N}}
\newcommand{\IA}{{\mathbb A}}

\newcommand{\R}{{\mathbb R}}

\newcommand{\err}{\mathrm{err}}

\theoremstyle{plain}
\newtheorem{theorem}{Theorem}[section]
\newtheorem{prop}[theorem]{Proposition}
\newtheorem{lemma}[theorem]{Lemma}

\newtheorem{defi}[theorem]{Definition}

\newtheorem{exa}[theorem]{Example}
\newtheorem{rem}[theorem]{Remark}

\theoremstyle{definition}

\makeatletter
\@namedef{subjclassname@2020}{%
	\textup{2020} Mathematics Subject Classification}
\makeatother

\begin{document}

\title[Existence of optimal residual ReLU neural networks]
{On the existence of minimizers in shallow residual\\ ReLU neural network optimization landscapes}

\author[]
{Steffen Dereich}
\address{Steffen Dereich\\
	Institute for Mathematical Stochastics\\
	Faculty of Mathematics and Computer Science\\
	University of M\"{u}nster, Germany}
\email{steffen.dereich@uni-muenster.de}

\author[]
{Arnulf Jentzen}
\address{Arnulf Jentzen\\
    School of Data Science and Shenzhen Research Institute of Big Data\\
    The Chinese University of Hong Kong, Shenzhen (CUHK-Shenzhen), China; Applied Mathematics: Institute for Analysis and Numerics\\
    Faculty of Mathematics and Computer Science\\
    University of M\"{u}nster, Germany}
\email{ajentzen@cuhk.edu.cn; ajentzen@uni-muenster.de}

\author[]
{Sebastian Kassing}
\address{Sebastian Kassing\\
	Faculty of Mathematics\\
	University of Bielefeld, Germany}
\email{skassing@math.uni-bielefeld.de}


\begin{abstract}
In this article, we show existence of minimizers in the loss landscape for residual artificial neural networks (ANNs) with multi-dimensional input layer 
and one hidden layer with  ReLU activation. Our work contrasts earlier results in~\cite{gallonjentzenlindner2022blowup} and~\cite{MR4243432} which showed that in many situations minimizers do not exist for common smooth activation functions even in the case where the target functions are polynomials. 
The proof of the existence property makes use of a closure of the search space containing all functions generated by ANNs and additional discontinuous \emph{generalized responses}. As we will show, the additional generalized responses in this larger space are suboptimal so that the minimum is attained in the original function class.  
\end{abstract}

\maketitle
\section{Introduction}
Machine learning methods -- 
often consisting of \emph{artificial neural networks} (ANNs) 
trained through \emph{gradient descent} (GD) type optimization methods 
-- are nowadays omnipresent computational methods 
which are heavily employed in many industrial applications 
as well as scientific research activities. 
Despite the mind-blowing success of such computational schemes, 
in general, it remains an open problem of research to 
rigorously prove (or disprove) the convergence of 
GD optimization methods in the training of ANNs. 
Even in the situation of shallow ANNs 
with just one hidden layer 
it remains an open research question 
whether GD methods do converge to a stationary point in the training of such ANNs.

In the literature regarding the training of ANNs there are, however, several partial error analysis results for GD type optimization methods 
(by which we mean, everything, time-continuous gradient flow processes, deterministic GD optimization methods, 
as well as stochastic GD optimization methods), see Section~\ref{sec:literature}. Many convergence results assume that the considered GD type optimization process is (almost surely) bounded, 
loosely speaking, in the sense that 
\begin{equation}
\label{eq:a_priori_bound}
\textstyle 
    \sup_{ t \in [0,\infty) }
    \| \Theta_t \|
    < \infty 
\end{equation}
where $ \Theta \colon [0,\infty) \to \R^{ \fd } $ 
corresponds to the employed GD type optimization process
(which could be a gradient flow optimization process 
or a time-continuous version of a discrete GD type optimization process), 
where $ \fd \in \N $ 
corresponds to the number of trainable parameters, 
and where $ \left\| \cdot \right\| $ refers to the standard norm on $ \R^{ \fd } $.

In general, it remains an open problem to verify \cref{eq:a_priori_bound} 
(and, thus, whether the many achievements concerning the convergence of GD are actually applicable in context of the training 
of ANNs).  
The question whether \cref{eq:a_priori_bound} is satisfied seems to be closely 
related to \emph{the existence of minimizers in the optimization landscape}; cf.\ \cite{gallonjentzenlindner2022blowup,MR4243432}. 
In particular, in \cite{gallonjentzenlindner2022blowup} 
counterexamples to \cref{eq:a_priori_bound} are given 
and \emph{divergence} of GD type optimization processes is proved  
in the sense that 
\begin{equation}
\label{eq:divergence}
  \liminf_{ t \to \infty }
  \| \Theta_t \|
  = \infty 
\end{equation} 
in certain cases where there do not exist minimizers in the optimization landscape. 
A divergence phenomenon of the form \cref{eq:divergence} 
may \emph{slow down} (or even completely rule out) the convergence of the error function, 
which is the highly relevant quantity in practical applications. In this aspect, it seems to be strongly advisable to 
\emph{design the ANN architecture} in a way so that there 
exist minimizers in the optimization landscape 
and so that the divergence in \cref{eq:divergence} \emph{fails to happen}. 

\emph{Overparametrized networks} in the setting of empirical risk minimization 
(more ReLU neurons than data points to fit) are able to perfectly interpolate the data (cf. \cite[Lemma~27.3]{foucart2022mathematical}) such that 
there exists a network configuration achieving zero error and, thus, a global minimum 
in the search space. 
For shallow feedforward ANNs using ReLU activation it has been shown that also in the underparametrized regime there exists a global minimum if the ANN has a one-dimensional output~\cite{le2023does}, whereas there are pathological counterexamples in higher dimensions~\cite{lim2022best}.

However, for general measures $ \mu $ not necessarily consisting of a finite number of Dirac measures 
the literature on the existence of global minima is very limited. The question whether there exist minimizers in the optimization landscape seems to be closely related to the choice of the \emph{activation function} 
and the specific architecture of the ANN. 
Indeed, in the case of \emph{fully-connected feedforward ANNs 
with just one hidden layer and one-dimensional input and output layer}, 
on the one hand 
for several common (smooth) activations 
such as 
\begin{itemize}
\item
the 
\emph{standard logistic activation}, 
\item 
the \emph{softplus activation}, 
\item 
the \emph{inverse tangent (arctan) activation}, 
\item 
the \emph{hyperbolic tangent activation}, 
and 
\item 
the \emph{softsign activation}
\end{itemize}
it has been shown 
(cf.\ 
\cite[Theorems~1.3 and 1.4]{gallonjentzenlindner2022blowup}
and 
\cite[Section~1.2]{MR4243432}) 
that, in general, there \emph{do not exist minimizers in the optimization landscape} 
even if the class of considered target functions 
is restricted to smooth functions or even polynomials 
but, on the other hand for the \emph{ReLU activation},  
it has been proved 
in \cite[Theorem~1.1]{jentzen2021existence} 
that for every Lipschitz continuous target function 
there \emph{do exist minimizers in the optimization landscape}. 
 These existence/non-existence phenomena for minimizers 
	in the optimization landscape thus reveal a \emph{fundamental difference}
	of the ReLU activation compared to the above mentioned 
	smooth activations. 
	This might give a partial explanation for why the ReLU activation seems to outperform other smooth activations in many regression tasks using a convex loss function, even though 
	it fails to be differentiable in contrast to the above mentioned 
	continuously differentiable activation functions.

Theorem~1.1 in \cite{jentzen2021existence} is, however, 
restricted to Lipschitz continuous target functions, 
to the \emph{standard mean square loss}, 
and to ANNs with one neuron on the input layer. ANNs with multi-dimensional input layer are not considered. 
In this work, we show existence   for \emph{shallow residual ANNs} with \emph{multi-dimensional input, general continuous target functions, and general strictly convex loss functions}.  As activation function we consider the \emph{ReLU activation} 
$ \R \ni x \mapsto \max\{ x, 0 \} = x^+ \in \R $. 
Interestingly, the multi-dimensionality has a significant impact on the problem. It induces an additional geometric assumption that has not appeared in the one-dimensional setting. This assumption is not an artifact of our approach, but a general prerequisite for the existence of minimizers. We provide a counterexample where the respective assumption is not satisfied and where no minimizers exist, see Example~\ref{exa:counter}.


We provide the main existence result of our work in two versions. \cref{thm:main} is the most general version. A slightly more restrictive but more intuitive version   is   \cref{thm:main_simplified0} below. There, we restrict to $L^{p}$-loss and work with a simpler geometric assumption compared with the general version.

%

Let us now formally introduce  shallow residual ANNs with $d_{\mathrm{in}}\in\N$ neurons on the input layer ($d_{\mathrm{in}}$ being fixed throughout the article) and $d\in\N_{0}$ neurons on the hidden layer. These can be parametrized by elements
\begin{equation}
	\label{eq:structurized}
	\IW=(W^1, W^2,b) \in \R^{d_{\mathrm{in}}\times (d+1)}\times \R^{d} \times \R^{d+1}
	= \colon \cW_d,
\end{equation}
where we enumerate the elements of the matrix and vectors as
\begin{align}
W^1=(w^{1}_{i,j})_{i=1,\dots,d_{\mathrm {in}},\,j=0,\dots,d}, \ W^2 = (w_j^2)_{j=1, \dots, d}  \text{ \ and \ }  b=(b_j)_{j=0, \dots, d}.
\end{align}
Moreover, for every $ j \in \{ 0, 1, \dots, d \} $ 
we write 
$
w_j^1 =  
(w_{ i, j }^1)_{i=1,\dots, d_{\mathrm{in}}}
\in \R^{ d_{ \mathrm{in} } }
$. 
We call $\IW$ a \emph{network configuration} and $\cW_d$ the \emph{parametrization class}.
We often refer to a configuration of a neural network as the (\emph{neural}) \emph{network}~$\IW$. 
A configuration $ \IW \in \cW_d $ is associated with the
function 
$
\mathfrak N^\IW \colon \R^{ d_{ \mathrm{in} } } \to \R
$ 
given by
\begin{align}
	\label{eq:response}
	\mathfrak N^\IW(x) 
	=   w_0^1 \cdot x   +b_{0} + 
	\sum_{j=1}^d w_{j}^2 
	\bigl( w_{j}^1 \cdot x + b_j \bigr)^{ \! + },
\end{align}
where $\cdot$ denotes the scalar product.
We call $\mathfrak N^\IW$ \emph{realization function} or \emph{response} of the network~$ \IW $. Note that in general the response of a network is a continuous and piecewise affine function from  
$\R^{ d_{ \mathrm{in} } }$ to $\R$. 
We conceive $\IW \mapsto \mathfrak N^{\IW}$ as a parametrization of a class of potential response functions  
$ 
\{ \mathfrak N^{\IW} \colon \IW \in \cW_d \} 
$ 
in a minimization problem.

Intuitively, $W^1$ describes the linear part (\emph{weight matrix}) in the affine transformation 
from the $ d_{ \mathrm{in} } $-dimensional input layer 
to the $ d $-dimensional hidden layer as well as in the skip connection, $W^2$ describe the linear part (weight matrix) in the affine transformation 
from the $ d $-dimensional hidden layer 
to the $ 1 $-dimensional output layer and $b$ describes the additive part (\emph{bias vector})
in the affine transformation 
from the $ d_{ \mathrm{in} } $-dimensional input layer 
to the $ d $-dimensional hidden layer as well as in the skip connection. 

%

It follows the simple version of our existence result.

\begin{theorem}[Existence of minimizer -- residual ANNs]
\label{thm:main_simplified0}
Let $ p \in (1,\infty) $, $ d_{\mathrm{in}} \in \N $, $ d \in \N_{0} $, 
let 
$ f \colon \R^{ d_{\mathrm{in}} } \to \R $ 
and 
$ h \colon \R^{ d_{\mathrm{in}} } \to [0,\infty) $
be continuous, 
assume that 
$
  h^{ - 1 }( (0,\infty) )
$
is a bounded and convex set, 
and let 
$
  \mathrm{err}^{p} \colon \mathcal W_d
  \to 
  \R
$
satisfy 
for all 
$
  \mathbb W 
  \in 
  \mathcal W_d
$
that 
\begin{align} 
  \int_{ 
    \R^{ 
      d_{\mathrm{in}} 
    } 
  }
  |
\textstyle 
    f(x)
    -
    \mathfrak{N}^{ \IW }( x )
  |^p\,
  h(x)
  \,
  \d x
\end{align}
where $ \mathcal W_d $ is as in \eqref{eq:structurized}
and where $ \mathfrak{N}^{ \IW } $ is as in \eqref{eq:response}.
 Then there exists a $ \IW' \in \mathcal W_d $
such that 
$
  \mathrm{err}^{p}( \IW' )
  =
  \inf_{ \IW \in \mathcal W_d }
  \mathrm{err}^{p}( \IW )
$. 
\end{theorem}

Loosely speaking, 
\cref{thm:main_simplified0} 
reveals in the situation of 
\emph{shallow residual ReLU ANNs} 
with a $ d_{ \mathrm{in} } $-dimensional input layer 
(with $ d_{ \mathrm{in} } $ neurons on the input layer), 
with a $ d $-dimensional hidden layer 
(with $ d $ neurons on the hidden layer), 
and with a skip-connection from the input layer 
to the output layer 
that there exist minimizers in the loss optimization landscape for the $L^p$-loss.

\cref{thm:main_simplified0} is a direct consequence
of the more general 
\cref{thm:main} below.
For this, let $ \mu \colon \mathcal{B}( \R^{ d_{ \mathrm{in} } } ) \to [0,\infty) $ 
be a finite measure on 
the Borel sets of $ \R^{ d_{ \mathrm{in} } } $, 
let $ \ID = \mathrm{supp}(\mu) $ be the support of $ \mu $, 
and 
let $ \mathcal L \colon \ID \times \R \to [0,\infty) $ 
be a product measurable function, the \emph{loss function}. 
We aim to minimize the error
\begin{equation}
  \mathrm{err}^{ \cL }( \IW ) 
  =
  \int_{ \ID } 
  \mathcal{L}( x, \mathfrak{N}^\IW( x ) ) \, \dd \mu(x)
\end{equation}
over all $\IW \in\cW_d$ for a given $d\in\N_0$
and we let
\begin{align} 
\label{eq_min}
\mathrm{err}^{\cL}_{d}=\inf_{\IW \in \cW_d}\mathrm{err}^{\cL}(\IW)
\end{align}
the minimal error with $ d $ neurons on the hidden layer. We stress that if there does not exist a neural network $\IW \in \cW_d$ satisfying $\err^\cL(\IW) = \err_d^\cL$ then every sequence $(\IW_n)_{n \in \N}\subseteq \cW_d$ of networks satisfying $\lim_{n \to \infty} \err^\cL(\IW_n) = \err_d^\cL$ diverges to infinity.

It follows  the main result.

\begin{theorem}[Existence of minimizers -- general loss functions and fully-connected residual ANNs] 
\label{thm:main}
Assume that $ \ID = \mathrm{supp}(\mu) $ is compact, 
assume that $ \mu $ has a continuous Lebesgue density $ h \colon \R^{ d_\mathrm{in} } \to [0,\infty) $, 
assume that 
for every hyperplane $ H \subseteq \R^{ d_{ \mathrm{in} } } $ intersecting the interior of the convex hull of $ \ID $ 
there is an element $ x \in H $ with $ h(x) > 0 $, 
and 
assume that the loss function $ \cL \colon \ID \times \R \to [0,\infty) $ 
satisfies the following assumptions:
\begin{enumerate}
\item[(i)]
(Continuity in the first argument) 
For every $ y \in \R $ it holds that $ \ID \ni x \mapsto \cL(x,y) \in \R $ is continuous. 
\item[(ii)]
(Strict convexity in the second argument) 
For all $ x \in \ID $ it holds that $ \R \ni y \mapsto \cL(x,y) \in \R $ is strictly convex and attains its minimum. 
\end{enumerate}
Then it holds for every $ d \in \N_0 $ that there exists an optimal network $ \IW\in\cW_d $ 
with $ \mathrm{err}^{\cL}( \IW ) = \mathrm{err}^{ \cL }_d $.
\end{theorem}

\cref{thm:main} is an immediate consequence of \cref{prop:minimal} below. In the following example we apply the main theorem for regression problems. In particular, the arguments entail validity of \cref{thm:main_simplified0}.

\begin{exa}[Regression problem]
Let $ \mu $ be as in \cref{thm:main}, 
let $ f \colon \R^{ d_\mathrm{in} } \to \R $ be continuous, 
and let $ L \colon \R \to [0,\infty) $ be a strictly convex function that attains its minimum. 
Then the function $ \cL \colon \R^{ d_\mathrm{in} } \times \R \to [0,\infty) $ given by
\begin{equation}
  \cL(x,y)=L(y-f(x))
\end{equation}
for all $ x \in \R^{ d_{ \mathrm{in} } } $ and $ y \in \R $ satisfies the assumptions of \cref{thm:main} 
and \cref{thm:main} allows us to conclude that the infimum 
\begin{equation}
  \inf _{\IW\in\cW_d} \int L(\mathfrak N^\IW(x)-f(x)) \,\dd \mu(x)
\end{equation}
is attained for a network  $\IW\in\cW_d$.
\end{exa}

	Our proof technique is different from the one in~\cite{jentzen2021existence}. In~\cite{jentzen2021existence}, it is shown that for the approximation of a one-dimensional Lipschitz target function $f$ on the interval $[0,1]$ one can restrict attention to ANNs having a response whose Lipschitz constant is bounded by $d \|f\|_{\mathrm{Lip}[0,1]}$, where $d$ denotes the number of neurons on the hidden layer. Since the latter is a compact space one can deduce existence of global minimizers.
	In this work, we first propose a closure of the space of response functions consisting of all network responses and additional limiting functions that are discontinuous. In a second step, we apply convexity arguments to construct for each discontinuous function in the functions space a network response that performs strictly better in the approximation of $f$. Both steps work in arbitrary input dimension, for arbitrary (possible non Lipschitz continuous) target functions $f$ and only uses the strict convexity of the loss function.

We note that the set of the realization functions 
of shallow residual ReLU ANNs with $ d_{ \mathrm{in} } $ neurons 
on the input layer, with $ d $ neurons on the hidden layer, 
and with a skip-connection from the input layer to the output layer coincides with the set of the realization functions of
shallow fully-connected feedforward ANNs with $ d_{\mathrm{in}} $ neurons on the input layer, 
with $ d + 1 $ neurons on the hidden layer, 
with the ReLU activation for $ d $ neurons on the hidden layer, 
and with the identity activation for the remaining neuron on the hidden layer. Moreover, the class of responses $ \{ \mathfrak N^{\IW} \colon \IW \in \cW_d\} $ 
that is considered in this article clearly contains all responses of shallow ANNs having only $d$ ReLU neurons on the hidden layer and no skip connection.
Conversely, since an affine function $ \mathfrak a \colon \R^{ d_{ \mathrm{in} } } \to \R $ 
can be represented as the response of two ReLU neurons, 
$ \{ \mathfrak N^{\IW} \colon \IW \in \cW_d\} $ is also contained in the class of responses of shallow ANNs
using $d+2$ neurons on the hidden layer that all apply the ReLU activation function. 
If $\mu$ is compactly supported, then the $\ID$-restricted  responses in $ \{ \mathfrak N^{\IW} \colon  \IW \in \cW_d\} $ even can be expressed by shallow ReLU networks with $d+1$ neurons.
	
The remainder of this article is organized as follows. In Section~\ref{sec:literature}, we summarize properties of the loss landscape in supervised learning and convergence statements for GD methods known in the literature. We stress that the analysis of the loss landscape is of particular interest since it effects the dynamics of all applied optimization methods. In Section~\ref{sec:genres}, we propose a closure of the space of response functions which contains all responses of shallow ReLU ANNs and additional discontinuous limit functions. In Section~\ref{sec:opt}, we show that the newly added functions are suboptimal. Moreover, we give a counterexample for Theorem~\ref{thm:main_simplified0} in the case that $h^{-1}((0,\infty))$ is a non-convex set, see Example~\ref{exa:counter}. We also give a necessary and sufficient condition that guarantees that the minimal error in the regression task strictly decreases after adding a ReLU neuron to the network architecture, see Proposition~\ref{cor:m} and Example~\ref{exa:counter2}. 

\section{Related work} \label{sec:literature}

\textbf{The loss landscape in supervised learning:} The analysis of the loss landscape in supervised learning is of special interest since its structure is important regardless of the applied optimization algorithm.
In the overparametrized regime, one can use Sard's theorem to show that, for smooth activation functions, the set of global minima forms for almost all input data a smooth manifold \cite{cooper2021global}. In~\cite{dereich2022}, the positive homogeneity of the ReLU activation is used to show that the configurations of shallow ReLU ANNs having the same response function is a smooth manifold, if one uses the minimal necessary number of neurons on the hidden layer in order to construct the particular response. This work also gives a sufficient condition for piecewise affine functions which implies that they can be written as responses of shallow ReLU ANNs. \cite{liu2021understanding} shows that, in the situation of shallow ReLU ANNs and finitely many data points, differentiable local minima are optimal within their cell of fixed activation pattern. They also give a criterion for when there exist non-differentiable local minima on the border of those cells.   
For the approximation of affine target functions one is able to completely classify the critical points and show the abstinence of local maxima \cite{cheridito2022landscape}.
For statements about the existence of and convergence to \emph{non-optimal local minima} in the training of (shallow) networks 
we refer the reader, e.g., 
to \cite{swirszcz2016local, safran2018spurious, Venturi2019Spurious,christof2022omnipresence,GentileWelper2022arXiv,Ibragimov2022arXiv}. 
See \cite{li2018visualizing} for visualizations of the loss landscape for various choices of network architectures. A good literature review regarding the loss landscape in neural network training can be found in \cite{E2020towards}. See also \cite{kainen2003best} for a result concerning the existence of a global minimum in the regression task of approximating functions in the space $ L^p( [0,1]^d) $ 
with shallow ANNs using \emph{heavyside activation}. Moreover, we refer to \cite{singerbest} for a general introduction into best approximators in normed spaces.  

\textbf{Convergence results for GD type optimization methods:} Several statements on the convergence of GD type optimization methods under the assumption (\ref{eq:a_priori_bound}) have been obtained in the literature. 
We refer the reader to
\cite{bolte2007lojasiewicz, Eberle2021arXiv, jentzen2021existence}
for results concerning gradient flows, 
to \cite{MR2197994, attouch2009convergence} for results concerning deterministic gradient methods,
to \cite{TADIC2015convergence,MR4056927, mertikopoulos2020sure, dereich2021convergence} for results concerning stochastic gradient methods
and
to \cite{dereich2021cooling} for results concerning gradient based diffusion processes.
Many of these results exploit \emph{Kurdyka-\L ojasiewicz} gradient type inequalities 
to establish convergence to a point (often a critical point) of the considered GD type optimization methods and we refer the reader to \cite{lojasiewicz1963propriete,lojasiewicz1965ensembles,lojasiewicz1984trajectoires} for classical results by \L ojasiewicz concerning 
gradient inequalities for analytic target functions and direct consequences for the convergence of gradient flows. In order to achieve (\ref{eq:a_priori_bound}), weight regularization can also be used to keep (stochastic) gradient descent from diverging, see~\cite[Lemma D.1]{dereich2023central}.

In the overparametrized setting, the stochastic noise in single-batch or mini-batch GD optimization methods vanishes when approaching the global minimum. This fact can be used to derive exponential convergence rates, see~\cite{wojtowytsch2023stochastic, gess2023convergence}.
For sophisticated convergence analysis for GD optimization methods in overparametrized regimes we point, e.g., to \cite{ChizatBach2018arXiv,Duetal2019arXivOverparametrized,Duetal2019arXivDeepOverparametrized,ChizatBach2020arXiv,Wojtowytsch2020arXiv} 
and the references therein.

\section{Generalized response of neural networks} \label{sec:genres}
We will work with more intuitive  geometric descriptions of realization functions of networks $\IW\in \cW_d$. We slightly modify the ideas given in \cite{dereich2022} and later introduce the notion of a \emph{generalized response}. Then, we show in Proposition~\ref{prop:genres} that, in very general approximation problems, there always exists a generalized response that solves the minimization task at least as well as the class $ \{ \mathfrak N^{\IW} \colon \IW \in \cW_d\} $.

We call a network $\IW\in\cW_d$ \emph{non-degenerate} iff 
for all $ j \in \{ 1, 2, \dots, d \} $ 
we have $ w_j^1 \not= 0 $. 
For a non-degenerate network $ \IW $, 
we say that the neuron $ j \in \{ 1, 2, \dots, d \} $  has
\begin{itemize} 
\item 
\emph{normal} 
$
  {\displaystyle 
  \vecn_j = 
  | w_j^1 |^{ - 1 } w_j^1 
  \in  \mathbb S^{d_{ \mathrm{in}-1}} 
  := \{ x \in \R^{ d_{ \mathrm{in} } } \colon |x| = 1 \} 
  } 
$,
\item \emph{offset} 
$
  o_j= - | w_j^1 |^{ - 1 } b_j \in \R
$, 
and
\item 
\emph{kink} 
$
  \Delta_j = | w_j^1 |  w_j ^2 \in \R
$. 
\end{itemize}
Moreover, we call the affine mapping $\mathfrak a \colon \R^{ d_{ \mathrm{in} } }\to \R$ given by
\begin{equation}
  \mathfrak a(x)=w_{0}^1 \cdot x+b_0
\end{equation}
\emph{affine background}. We call $(\vecn, o, \Delta,  \mathfrak a)$ with $\vecn = (\vecn_1, \dots, \vecn_{d})\in  (\mathbb S^{d_{ \mathrm{in}-1}})^d$, $o=(o_1, \dots, o_d)\in \R^d$, $\Delta=(\Delta_1, \dots, \Delta_d)\in\R^d$ and the affine function $\mathfrak a$ the \emph{effective tuple of $\IW$} and write $\cE_{d}$ for the set of all effective tuples using $d$ ReLU neurons.

First, we note that the response of a non-degenerate ANN $ \IW $ 
can be represented in terms of its effective tuple:
\begin{equation}
\begin{split}
\mathfrak N^\IW(x)&=\mathfrak a(x)+\sum_{j=1}^{d} w^2_j(w^1_j \cdot x+b_j)^{ + }=\mathfrak a(x)+\sum_{j=1}^{d}\Delta_j \Bigl(\frac 1{|w_j^1|} w^1_j \cdot x+\frac 1{|w_j^1|} b_j\Bigr)^{ + }\\
&=\mathfrak a(x) + \sum_{j=1}^{d} \Delta_j \bigl( \vecn_j \cdot x- o_j\bigr)^{ + }.
\end{split}
\end{equation} 
With slight misuse of notation we also write
\begin{equation}
  \mathfrak{N}^{\vecn, o, \Delta,\mathfrak a} \colon \R^{ d_{ \mathrm{in} } }\to \R, 
  \qquad 
  x\mapsto \mathfrak a(x)+ \sum_{j=1}^{d} \Delta_j \bigl( \vecn_j \cdot x- o_j\bigr)^{ + }
\end{equation}
and $\err^{\cL}(\vecn,o, \Delta, \mathfrak a)= \int \cL(x, \mathfrak N^{\vecn, o, \Delta, \mathfrak a}(x)) \, \dd \mu(x).$  
Although the tuple $(\vecn,o,\Delta,\mathfrak a)$ does not uniquely describe a neural network, it fully characterizes the response function and thus we will speak of the neural network with effective tuple  $(\vecn,o,\Delta,\mathfrak a)$.

We stress that also the response of a degenerate network $\IW$ can be  described as response associated to an effective tuple. Indeed, for every $j\in \{1, \dots, d\} $ with $w_j^1=0$  the respective neuron has a constant contribution  $w_j^2(b_j)^{ + }$. Now, one can choose an arbitrary normal $\vecn_j$ and offset $o_j$, set the kink equal to zero ($\Delta_j=0$) and add the constant $w_j^2(b_j)^{ + }$ to the affine background $\mathfrak a$. Repeating this procedure for every such neuron we get an effective tuple $(\vecn, o, \Delta, \mathfrak a) \in \cE_d$ that satisfies $\mathfrak N^{\vecn, o, \Delta, \mathfrak a} = \mathfrak N^{\IW}$. Conversely, for every effective tuple $(\vecn, o, \Delta,\mathfrak a) \in \cE_d$, the mapping
$\mathfrak N^{\vecn,o,\Delta,\mathfrak a}$ is the response of an appropriate network $\IW \in \cW_d$. In fact, for $j=1,\dots, d$,
one can choose $w_j^1=\vecn_j$, $b_j=-o_j$ and $w_j^2=\Delta_j$ and gets that  for all $x \in \R^{ d_{ \mathrm{in} } }$
\begin{equation}
	w_j^2(w_j^1 \cdot x + b_j)^{ + } = \Delta_j (\vecn_j \cdot x-o_j)^{ + }.
\end{equation}
Analogously,
one can choose $w_0^1 = \mathfrak a'$ and $b_0=\mathfrak a(0)$ such that for all $x \in \R^{ d_{ \mathrm{in} } }$
\begin{equation}
  \mathfrak a (x) = w_0^1 \cdot x + b_0.
\end{equation}
This entails that
\begin{equation}
  \mathrm{err}^{\cL}_d = \inf_{(\vecn,o,\Delta,\mathfrak b)\in\cE_d}\int \cL(x,\mathfrak N^{\vecn,o,\Delta,\mathfrak b}(x))\, \dd\mu(x)
\end{equation} 
and the infimum is attained iff there is a network $\IW\in\cW_d$ for which the infimum in~(\ref{eq_min}) is attained.

For an effective tuple $(\vecn, o, \Delta, \mathfrak a) \in \cE_d$,  we say that the $j$th ReLU neuron
has the \emph{breakline}
\begin{equation}
  H_j = \bigl\{ x \in \R^{ d_{ \mathrm{in} } } \colon \vecn_j \cdot x =  o_j \bigr\}
\end{equation}
and we call 
\begin{equation}
  A_j = \{ x \in \R^{ d_{ \mathrm{in} } } \colon \vecn_j \cdot x > o_j \}
\end{equation}
the \emph{domain of activity} of the $j$th ReLU neuron.
By construction, we have 
\begin{equation}
  \mathfrak{N}^{\vecn,o,\Delta,\mathfrak a}(x)
  = 
  \mathfrak{a}(x) 
  + 
  \sum_{j=1}^d 
  \bigl(
    \Delta_j( \vecn_j \cdot x - o_j )
  \bigr) 
  \1_{ A_j }( x ) 
  .
\end{equation}
Outside the breaklines the function 
$
  \mathfrak{N}^{ \vecn, o, \Delta, \mathfrak{a} } 
$ 
is differentiable with
\begin{equation}
  D \mathfrak{N}^{\vecn,o,\Delta,\mathfrak a}( x ) 
  =
  \mathfrak{a}'(x) 
  + 
  \sum_{ j = 1 }^d 
  \Delta_j \vecn_j 
  \1_{ A_j }( x ) 
  .
\end{equation}
Note that for each summand $j=1, \dots, d$ along the breakline the difference of the differential on $A_j$ and $\overline{A}_j^c$ equals $\Delta_j \vecn_j$ (which is also true for the response function $\mathfrak N^\IW$ provided that it is differentiable in the reference points and there does not exist a second neuron having the same breakline $H_j$).


For a better understanding of the optimization problem discussed in this article it makes sense to view the set of responses $ \{ \mathfrak N^\IW \colon \IW\in \cW_d\} $ as a subset  of the locally convex vector space $\cL^1_\mathrm{loc}$ of  locally integrable functions. Then the set of response functions is not closed in $\cL^1_\mathrm{loc}$ and the main task of this article is to show that functions in $\cL^1_\mathrm{loc}$ that can be approached by functions from $ \{ \mathfrak N^\IW \colon \IW\in \cW_d\} $ but are itself no response functions will provide larger errors than  the response functions.  


We now extend the family of response functions.
\begin{defi}  We call a function $\cR \colon \R^{ d_{ \mathrm{in} } } \to \R$ a \emph{generalized response} if it admits the following representation: there are
$K\in\N_0$,  a tuple of open half-spaces $\mathbf A=(A_1,\dots,A_K)$ of $\R^{ d_{ \mathrm{in} } }$ with pairwise distinct boundaries 
$\partial A_1,\dots,\partial A_K$, 
a vector $\mathbf m =(m_1,\dots,m_K)\in\{1,2\}^K$, an affine mapping $\mathfrak a \colon \R^{ d_{ \mathrm{in} } } \to \R$, 
vectors $\delta_1, \dots, \delta_K \in \R^{ d_{ \mathrm{in} } }$, 
and reals $\mathfrak b_1, \dots, \mathfrak b_K \in \R$ such that 
\begin{enumerate}
 \item[(i)]
it holds for all $x \in \R^{ d_{ \mathrm{in} } }$ that 
\begin{align}
\label{eq84529}
  \cR(x) 
  = \mathfrak{a}( x ) 
  + 
  \sum_{ k = 1 }^K 
  \bigl(
    \delta_k \cdot x + \mathfrak{b}_k
  \bigr) 
  \1_{ A_k }( x ) 
\end{align}
and 
\item[(ii)]
it holds for all $ k \in \{ 1, \dots, K \} $ with $ m_k = 1 $ that 
\begin{equation}
  \partial A_k \subseteq \{ x \in \R^{ d_{ \mathrm{in} } } \colon \delta_k\cdot x+\mathfrak b_k = 0 \}.
\end{equation}
\end{enumerate}
We will represent generalized responses as in \cref{eq84529}, 
we call  $ A_1, \dots, A_K $ 
\emph{the active half-spaces} of the response, 
and we call $ m_1, \dots, m_K $ 
the \emph{multiplicities} of the half-spaces $ A_1, \dots, A_K $. 
The minimal number  $m_1+\ldots+m_K$ that can be achieved in such a representation is called the \emph{dimension} of~$\cR$. 
For every $d\in\N_0$ we denote by $\mathfrak R_d$ the family of all generalized responses of dimension $d$ or smaller. We call a generalized response \emph{simple} if it is continuous which means that all multiplicities can be chosen equal to one. A response~$\cR\in\mathfrak R_d$ is called \emph{strict at dimension $d$}  if the response has dimension $d-1$ or  is discontinuous.
We denote by   $\mathfrak R_d^\mathrm{strict}$ the responses in $\mathfrak R_d$ that are strict at dimension $d$.
\end{defi}

\begin{rem}
Note that the sets $ \{ \mathfrak N^\IW \colon \IW\in \cW_d\} $ and $ \{ \cR \colon \cR\in \mathfrak R_d\text{ is simple}\} $ agree. 
\end{rem}

The condition $\partial A_k \subseteq \{ x \in \R^{ d_{ \mathrm{in} } } \colon \delta_k\cdot x+\mathfrak b_k = 0 \}$ is equivalent to the condition that $x\mapsto  \bigl(\delta_{k}\cdot x+\mathfrak b_{k}\bigr)\1_{A_k}(x)$ is continuous which, in particular, implies that $\delta_k \perp \partial A_k$. The next remark shows that every generalized response $\mathcal R \colon \R^{ d_{ \mathrm{in} } }\to \R$ of dimension $d$ or smaller (even the generalized responses with $\delta_k \not\perp \partial A_k$ for some $k$) is, on $ \R^{ d_{ \mathrm{in} } } \backslash ( \bigcup_{k=1, \dots, K} \partial A_k ) $, 
the limit of network responses of networks in $\cW_d$.

\begin{rem}[Asymptotic ANN representations for generalized responses] \label{rem:genres}
	Let $\vecn\in \mathbb S^{d_{ \mathrm{in}-1}}$, $\delta\in\R^{ d_{ \mathrm{in} } }$ and $o,\mathfrak b\in\R$ and set
	\begin{equation}
		A=\{x\in\R^{ d_{ \mathrm{in} } } \colon  \vecn\cdot x>o\}
		\qquad
		\text{and}
		\qquad 
		\forall \, x \in \R^{ d_{ \mathrm{in} } } \colon
		\cR(x) =
		( \delta \cdot x + \mathfrak{b} ) \1_A(x)
		.
	\end{equation}
	We will show that the generalized response $ \cR $ 
	is on $ \R^{ d_{ \mathrm{in} } } \backslash ( \partial A ) $ 
	the limit of the response of two ReLU neurons. For every $n\in\N$, the following function 
	\begin{equation}
		\cR_n(x) = 
		\frac 12\bigl((\delta+n\vecn)\cdot x+\mathfrak b-no\bigr)^{ + } 
		- \frac 12\bigl((-\delta+n\vecn)\cdot x-\mathfrak b-no\bigr)^{ + }
	\end{equation}
	is the response of a shallow ReLU ANN with two hidden ReLU neurons. 
	Now note that for all $ x \in A $ 
	one has $\vecn\cdot x-o>0$ so that
	\begin{equation}
		(\delta+n\vecn)\cdot x+\mathfrak b-no=n (\vecn\cdot x-o)+\delta\cdot x+ \mathfrak b\to \infty.
	\end{equation}
	Analogously, $(-\delta+n\vecn)\cdot x-\mathfrak b-no\to \infty$. Consequently,  there exists an $N \in \N$ depending on $x$ such that for all $n \ge N$
	\begin{equation}
		\cR_n(x)= \delta\cdot x+\mathfrak b =\cR(x).
	\end{equation}
	Conversely, for all $ x \in {\overline A}^c $ 
	one has $ \vecn \cdot x - o < 0 $ 
	so that analogously to above
	$
	(\delta+n\vecn)\cdot x+\mathfrak b-no\to-\infty
	$
	and  
	$
	(-\delta+n\vecn)\cdot x-\mathfrak b-no\to -\infty
	$. 
	Consequently,  there exists an $N \in \N$ depending on $x$ such that for all $n \ge N$
	\begin{equation}
		\cR_n(x)=0=\cR(x)
		.
	\end{equation}
	We thus represented $\cR$ as asymptotic response of a ReLU ANN 
	with two hidden neurons.

	For a  generalized response one replaces every term $(\delta_k\cdot x+\mathfrak b)\1_{A_k}(x)$ of multiplicity two (see~(\ref{eq84529})) by two ReLU neurons exactly as above. 
	Moreover, the terms with multiplicity one are responses of ANNs with one ReLU neuron. 
\end{rem}

In this section, we work with a general measure $ \mu $ which may have unbounded support. 
The assumptions on $ \mu $ are stated in the next definition. 

\begin{defi}\begin{enumerate}
\item[(i)] An element $ x $ of a hyperplane $ H \subseteq \R^{ d_{ \mathrm{in} } }$ is called \emph{$H$-regular} if $x\in \mathrm{supp}\, \mu|_A$ and $x\in \mathrm{supp}\, \mu|_{\overline A^c}$, where $A$ is an open half-space with $\partial A=H$.
\item[(ii)] A measure $\mu$ is called \emph{nice} if all hyperplanes have $\mu$-measure zero and if for every open half-space $A$ with $\mu(A),\mu(\overline A^c)>0$ the set of $\partial A$-regular points cannot be covered by finitely many hyperplanes different from $\partial A$.
\end{enumerate} 
\end{defi}

 Next, we will show that under quite weak assumptions there exist  generalized responses of dimension $d$  or smaller that achieve an  error of at most $\mathrm{err}_d ^\cL$. 

\begin{prop} \label{prop:genres} 
Assume that $ \mu $ is a nice measure on a closed subset $ \ID \subseteq \R^{ d_{ \mathrm{in} } } $ of $ \R^{ d_{ \mathrm{in} } } $  
and assume that the loss function $ \cL \colon \ID \times \R \to [0,\infty) $ is measurable and satisfies the following assumptions: 
		\begin{enumerate}
			\item [(i)] (Lower-semincontinuity in the second argument) For all $ x \in \ID $, $ y \in \R $ we have
			\begin{equation}
			\textstyle 
				\liminf_{y' \to y} \cL(x,y') \ge \cL(x,y).
			\end{equation}
			\item [(ii)] (Unbounded in the second argument) For all $ x \in \ID $ we have
			\begin{equation}
			\textstyle 
				\lim_{|y| \to \infty} \cL(x,y) = \infty.
			\end{equation}
		\end{enumerate}
Let $d\in\N_0$    with $\mathrm{err}_d^{\cL}<\infty$. 	Then there exists a generalized response $\cR\in\mathfrak R_d$ which satisfies 
\begin{equation}
  \int \cL(x, \cR(x))\,\dd\mu(x) =\overline{\mathrm{err}}^\cL_d:=  \inf_{\tilde \cR\in\mathfrak R_d} \int \cL(x, \tilde\cR(x))\,\dd\mu(x) .
\end{equation}
Furthermore, if $ d \ge 1 $, then the infimum
\begin{equation}
  \inf_{\tilde \cR \in \mathfrak R_d^{\mathrm{strict}}}  \int \cL(x, \tilde\cR(x))\,\dd\mu(x)
\end{equation}
is attained on $\mathfrak R_d^{\mathrm{strict}}$.
\end{prop}

\begin{proof}
Let $(\cR^{(n)})_{n \in \N}$ be a sequence of generalized responses in $\mathfrak R_d$ that satisfies
\begin{equation}
  \lim\limits_{ n \to \infty } 
  \int 
  \cL( x, \cR^{ (n) }( x ) ) \, 
  \dd \mu (x) 
  = 
  \overline{\mathrm{err}}_d^{ \cL }
  .
\end{equation}
We use the representations as in~\cref{eq84529} 
and write
\begin{equation}
  \cR^{ (n) }( x ) 
  =
  \mathfrak{a}^{ (n) }( x ) 
  +
  \sum_{ k = 1 }^{ K_n } 
  \bigl(
    \delta_k^{ (n) } \cdot x + \mathfrak{b}_k^{ (n) } 
  \bigr)
  \1_{ A_k^{ (n) } }( x )
  .
\end{equation}
Moreover, denote by 
$ \mathfrak{n}_k^{(n)} \in \mathbb S^{d_{ \mathrm{in}-1}} $ 
and 
$ o_k^{ (n) } \in \R $ 
the quantities with 
$
  A_k^{ (n) } 
  =
  \{ 
    x \in \R^{ d_{ \mathrm{in} } } \colon 
    \mathfrak{n}_k^{ (n) } \cdot x > o_k^{ (n) } 
  \} 
$ 
and by 
$
  \mathbf{m}^{ (n) } 
  = 
  ( m_1^{ (n) }, \dots, m_{ K^{(n)} }^{ (n) } ) 
$ 
the respective multiplicities.

\underline{1. Step:} Choosing an appropriate subsequence.

Since $ K^{(n)} \le d $ for all $ n \in \N $ 
we can choose a subsequence 
$
  ( \ell_n )_{ n \in \N } 
$ 
such that there exists $ K \in \N_0 $ with $ K^{ ( \ell_n ) } = K $ for all $ n \in \N $. 
For ease of notation we will assume that this is the case for the full sequence. 
With the same argument we can assume 
without loss of generality that 
there exists 
$ 
  \mathbf{m} = ( m_1, \dots, m_K ) \in \{ 1, 2 \}^K 
$ 
such that for all $ n \in \N $ we have 
$ 
  \mathbf{m}^{ (n) } = \mathbf{m} 
$.

Moreover, after possibly thinning the sequence again 
we can assure that for all $ k \in \{ 1, 2, \dots, K \} $ 
we have convergence 
$
  \mathfrak{n}^{ (n) }_k \to \mathfrak{n}_k 
$ 
in the compact space $\mathbb S^{d_{ \mathrm{in}-1}} $ and $ o_k^{ (n) } \to o_k $ 
in the two point compactification $ \R \cup \{ \pm \infty \} $. 
We assign each $ k \in \{ 1, 2, \dots, K \} $ 
an \emph{asymptotic active area} $ A_k $ 
given by
\begin{equation}
  A_k = 
  \{ 
    x \in \R^{ d_{ \mathrm{in} } } \colon \mathfrak{n}_k \cdot x > o_k 
  \}
\end{equation}
which is degenerate in the case where $ o_k \in\{ \pm \infty \} $.

We denote by $ H_k $ the respective breakline $ \partial A_k $. 
Even if, for every $ n \in \N $, the original breaklines $ \partial A_1^{ (n) }, \dots, \partial A_K^{ (n) } $ 
are pairwise distinct, this might not be true for the limiting ones. 
In particular, there may be several $ k $'s for which the asymptotic active areas may 
be on opposite sides of the same breakline. In that case we choose one side and 
replace for each~$ k $ with asymptotic active area on the opposite side 
its contribution in the 
representation \cref{eq84529} from 
$ 
  ( 
    \delta_k^{ (n) } \cdot x + \mathfrak{b}_k^{ (n) } 
  )
  \1_{ A_k^{ (n) } }( x ) 
$ 
to
\begin{equation}
  \bigl(
    - \delta_k^{ (n) } \cdot x - \mathfrak{b}_k^{ (n) } 
  \bigr)
  \1_{ ( \overline{A}_k^{ (n) } )^c }( x ) 
  + \delta_k^{ (n) } \cdot x 
  + \mathfrak{b}_k^{ (n) }
\end{equation}
which agrees with the former term outside the breakline (which is a zero set).
This means we replace $ \delta_k^{ (n) } $, 
$ \mathfrak{b}_k^{ (n) } $, 
and $ A_k^{ (n) } $ 
by 
$ - \delta_k^{ (n) } $, 
$ - \mathfrak{b}_k^{ (n) } $, 
and 
$
  \bigl( \overline{A}_k^{ (n) } \bigr)^c 
$, 
respectively, and adjust the respective affine background accordingly.
Thus we can assume without loss of generality that all asymptotic active areas 
sharing the same breakline are on the same side.

We use the asymptotic active areas to partition the space: 
let $ \IJ $ denote the collection of all subsets $ J \subseteq \{ 1, 2, \dots, K \} $ 
for which the set
\begin{equation}
\textstyle 
  A_J = 
  \bigl( 
    \bigcap_{ j \in J } A_j 
  \bigr) 
  \cap 
  \bigl( 
    \bigcap_{ j \in J^c } \overline A_j^c
  \bigr)
\end{equation}
satisfies $ \mu( A_J ) > 0 $. 
We note that the sets 
$
  A_J 
$, 
$ 
  J \in \IJ 
$, 
are non-empty, open, and pairwise disjoint 
and their union has full $ \mu $-measure since
\begin{equation}
\textstyle 
  \mu\Bigl( 
    \R^{ d_{ \mathrm{in} } } \backslash 
    \bigcup_{ J \subseteq \{ 1, 2, \dots, K \} } A_J 
  \Bigr) 
  \le 
  \sum\limits_{ j = 1 }^K 
  \mu\bigl( H_j \bigr) = 0 .
\end{equation}
Moreover, for every $ J \in \IJ $ and every compact set $ B $ with $ B \subseteq A_J $ 
one has from a $ B $-dependent $ n $ onward that 
the generalized response $ \cR^{ (n) } $ 
satisfies for all $ x \in B $ that
\begin{equation}
  \cR^{(n)}(x)= \cD_J ^{(n)} \cdot x +\beta_J^{(n)} ,
\end{equation}
where  
\begin{equation}
  \cD_J^{(n)} 
  := { \mathfrak{a}' }^{ (n) } 
  + \sum_{ j \in J } \delta_j^{ (n) } 
\qquad 
  \text{and}
\qquad 
  \beta_J^{ (n) } := \mathfrak{a}^{ (n) }( 0 ) 
  + \sum_{ j \in J } \mathfrak{b}_j^{ (n) } .
\end{equation}

Let $ J \in \IJ $. 
Next, we show that along an appropriate subsequence, 
we have convergence of 
$
  ( \cD_J^{ (n) } ) _{n \in \N}
$ 
in 
$
  \R^{ d_{ \mathrm{in} } } 
$. 
First assume that along a subsequence 
one has that $ ( | \cD_J^{ (n) } | )_{n \in \N} $ converges to $ \infty $. 
For ease of notation 
we assume without loss of generality that 
$
  | \cD_J^{ (n) } | \to \infty 
$. 
We let
\begin{equation}
  \cH_J^{ (n) } = 
  \{ 
    x \in \R^{ d_{ \mathrm{in} } } \colon \cD_J^{ (n) } \cdot x + \beta_J = 0 
  \} .
\end{equation}
For every $ n $ with $ \cD_J^{ (n) } \not= 0 $, $ \cH_J^{ (n) } $ 
is a hyperplane which can be parametrized by taking a normal and 
the respective offset. 
As above we can argue that along an appropriate subsequence 
(which is again assumed to be the whole sequence) 
one has convergence of the normals in $ \mathbb S^{d_{ \mathrm{in}-1}} $ and 
of the offsets in $ \R\cup\{ \pm \infty \} $. 
We denote by $ \cH_J $ the hyperplane being associated 
to the limiting normal and offset (which is assumed to be the empty set 
in the case where the offsets do not converge in $ \R $). 
Since the norm of the gradient $ \cD_J^{ (n) } $ tends 
to infinity we get that for every $ x \in A_J \backslash \cH_J $ 
one has $ | \cR^{ (n) }( x ) | \to \infty $ and, 
hence, $ \cL( x, \cR^{ (n) }( x ) ) \to \infty $. 
Consequently, Fatou implies that
\begin{equation}
\begin{split}
  \liminf\limits_{ n \to \infty } 
  \int_{ A_J \backslash \cH_J }  
  \cL( x, \cR^{ (n) }( x ) ) \, \dd \mu(x) 
&
  \ge 
  \int_{ A_J \backslash \cH_J } 
  \liminf\limits_{ n \to \infty } 
  \cL( x, \cR^{ (n) }( x ) ) 
  \, \dd \mu(x) 
  = \infty 
\end{split}
\end{equation}
contradicting the asymptotic optimality of $ ( \cR^{ (n) } )_{ n \in \N } $. 
We showed that the sequence $ ( \cD_J^{ (n) } )_{ n \in \N } $ 
is precompact and by switching to an appropriate subsequence 
we can guarantee that the limit 
$
  \cD_J = \lim_{ n \to \infty } \cD_J^{ (n) } 
$ exists.

Similarly we show that along an appropriate subsequence $ ( \beta_J^{ (n) } )_{ n \in \N } $ 
converges to a value $ \beta_J \in \R $. 
Suppose this were not the case, 
then there were a subsequence along which $ | \beta_J^{ (n) } | \to \infty $ 
(again we assume for ease of notation that this is the case along the full sequence). 
Then for every $ x \in A_J $, 
$ | \cR^{ (n) }( x ) | \to \infty $ and 
we argue as above that this contradicts 
the optimality of $ ( \cR^{ (n) } )_{ n \in \N } $. 
Consequently, we have for an appropriately thinned sequence 
on a compact set $ B \subseteq A_J $ uniform convergence 
\begin{equation}
\label{eq87314}
  \lim_{ n \to \infty } 
  \cR^{ (n) }( x ) 
  = \cD_J \cdot x + \beta_J .
\end{equation}

Since $ \bigcup_{ J \in \IJ } A_J $ has 
full $ \mu $-measure we get 
with the lower semicontinuity of $ \cL $ in the second argument and 
Fatou's lemma that for every measurable 
function $ \cR \colon \R^{ d_{ \mathrm{in} } } \to \R $ 
satisfying for each $ J \in \IJ $, $ x \in A_J $ 
that
\begin{equation}
\label{eq:DefR}
  \cR(x) = \cD_J \cdot x + \beta_J
\end{equation}
we have 
\begin{equation}
\begin{split}
  \int 
  \mathcal{L}( x, \cR(x) ) 
  \, \dd \mu(x) 
&
  = 
  \int 
    \liminf\limits_{ n \to \infty} \cL( x, \cR^{ (n) }( x ) ) 
  \, \dd \mu (x) 
\\
&
  \le 
  \liminf\limits_{ n \to \infty } 
  \int 
  \cL( x, \cR^{ (n) }( x ) ) 
  \, \dd \mu( x ) 
  = 
  \overline{\mathrm{err}}_{d}^{ \cL }
  .
\end{split}
\end{equation}

\underline{Step 2:} $ \cR $ may be chosen as 
a generalized response of dimension $ d $ or smaller.
    
We call a summand $ k \in \{ 1, 2, \dots, K \} $ 
degenerate if $ A_k $ or 
$ \overline{A}_k^c $ 
has $ \mu $-measure zero. 
We omit every degenerate summand $ k $ in the sense 
that we set $ \delta_k^{ (n) } = 0 $ and 
$ \mathfrak{b}_k^{ (n) } = 0 $ for all $ n \in \N $ 
and note 
that by adjusting the affine background appropriately 
we still have validity of \cref{eq87314} 
with the same limit on all relevant cells 
and, in particular, $ \mu $-almost everywhere. 
    
Let now $ k $ be a non-degenerate summand. 
Since $ \mu $ is nice there exists a $ \partial A_k $-regular point $ x $ that 
is not in $ \bigcup_{ A \in \IA \colon A \not= A_k } \partial A $, 
where $ \IA := \{ A_j \colon j \text{ is non-degenerate} \} $. 
We let 
\begin{equation}
  J^x_- 
  = 
  \{ j \colon x \in A_j \} 
\qquad 
  \text{and}
\qquad 
  J^x_+ = 
  J^x_- \cup \{ j \colon A_j = A_k \} 
  .
\end{equation}
Since $ x \in \mathrm{supp}( \mu|_{ \overline{A}_j^c } ) $ we get 
that the cell $ A_{ J^x_- } $ has strictly positive 
$ \mu $-measure so that $ J^x_- \in \IJ $. 
Analogously, 
$ x \in \mathrm{supp}( \mu|_{ A_j } ) $ entails 
that $ J^x_+ \in \IJ $. 
(Note that $ J_+^x $ and $ J_-^x $ are just the cells that 
lie on the opposite sides of the hyperplane $ \partial A_k $ at $ x $.) 
We thus get that 
\begin{equation}
  \delta_{A_k}^{(n)} 
  := 
  \sum_{ j \colon A_j = A_k } 
  \delta_j^{ (n) } 
  = 
  \cD_{ J_+^x }^{ (n) } 
  - 
  \cD_{ J_-^x }^{ (n) }
  \to \cD_{ J_+^x } - \cD_{ J_-^x } 
  =: 
  \delta_{ A_k } 
  ,
\end{equation}
where the definitions of $ \delta_{ A_k }^{ (n) } $ and 
$ \delta_{ A_k } $ do not depend 
on the choice of $ x $. 
Analogously,
\begin{equation}
  \mathfrak{b}_{ A_k }^{(n)} 
  := 
  \sum_{ j \colon A_j = A_k } 
  \mathfrak{b}_j^{ (n) } 
  = 
  \beta_{ J_+^x }^{ (n) } 
  - 
  \beta_{ J_-^x }^{ (n) } 
  \to 
  \beta_{ J_+^x }
  -
  \beta_{ J_-^x }
  =: 
  \mathfrak b_{A_k}
  .
\end{equation}
    
Now for general $ J \in \IJ $, 
we have
\begin{equation}
  \cD_J\leftarrow \cD_J^{(n)}= {\mathfrak a'}^{(n)} +\sum_{A\in \{A_j:j\in J\}} \delta_A^{(n)}
\qquad 
  \text{and}
\qquad 
  \beta_J\leftarrow \beta_J^{(n)}= {\mathfrak a}^{(n)}(0) +\sum_{A\in \{A_j:j\in J\}} \mathfrak b_A^{(n)}  
  . 
\end{equation}
Since 
$
  \sum_{ A \in \{ A_j \colon j \in J \} } \delta_A^{ (n) } 
$ 
and 
$
  \sum_{ A \in \{ A_j \colon j \in J \} } \mathfrak b_A^{ (n) }  
$ 
converge to $ \sum_{ A\in \{A_j \colon j \in J \} } \delta_A $ and  
$ \sum_{ A \in \{ A_j \colon j \in J \} } \mathfrak{b}_A $, 
respectively, we have that 
$
  ( {\mathfrak{a}'}^{ (n) } )_{ n \in \N } 
$ 
and 
$
  ( \mathfrak{a}^{ (n) }(0))_{ n \in \N }
$ 
converge and there is an appropriate affine function $ \mathfrak{a} $ 
such that for all $J\in\IJ$ and  $x\in A_J$
\begin{equation}
  \cR(x)
  = 
  \cD_J\cdot x+\beta_J
  = 
  \mathfrak{a}(x) 
  + 
  \sum_{ A \in \IA } 
  \bigl( \delta_A\cdot x +\mathfrak b_A \bigr)
  \1_A(x)
  .
\end{equation}
So far, we have only used the definition of $\cR$ on $\bigcup_{J\in\IJ}A_J$ and 
we now assume that $\cR$ is chosen in such a way that 
the latter identity holds for all $x\in \R^{ d_{ \mathrm{in} } }$ 
(by possibly changing the definition on a $\mu$-nullset).
We still need to show that $\cR$ is a generalized response of dimension $d$ or smaller.

Every active area $A\in\IA$ that is the asymptotic active area 
of a single non-degenerate summand $k$ with $m_k=1$ is assigned the multiplicity one. 
All other non-degenerate active areas get multiplicity two. 
Then the overall multiplicity (the sum of the individual multiplicities) is smaller or equal to the dimension~$d$. To see this recall that every  active area $A\in\IA$ that is the asymptotic area of more than one summand or one summand of multiplicity two also contributed at least two to the multiplicity   of the approximating responses. All degenerate summands do not contribute at all although they contribute to the approximating responses.

It remains to show that $\cR$ is indeed a generalized response with the active areas $\IA$ having the above multiplicities.  For this it remains to show continuity of
    $\1_A(x)(\delta_A\cdot x+\mathfrak b_A)$ for all $A\in\IA$ with assigned multiplicity one. Suppose that the $k$th summand is the unique summand that contributes to such an $A$. Then $\delta_k^{(n)}=\delta_A^{(n)}\to \delta_A$ and $\mathfrak b_k^{(n)}=\mathfrak b^{(n)}_A\to \mathfrak b_A$. Moreover, one has
\begin{equation}
    \{x\in\R^{ d_{ \mathrm{in} } } \colon \mathfrak{n}_k^{(n)}\cdot x- o_k^{(n)}=0\}\subseteq  \{x\in\R^{ d_{ \mathrm{in} } } \colon \delta _k^{(n)}\cdot x+ \mathfrak b_k^{(n)}=0\}
\end{equation}
    which entails that, in particular, $\delta_k^{(n)}$ is a multiple of $\mathfrak{n}_k^{(n)}$. Both latter vectors converge which also entails that the limit $\delta_A$ is a multiple of $\mathfrak{n}_k$.  To show that
\begin{equation}
    \partial A\subseteq \{x\in\R^{ d_{ \mathrm{in} } } \colon \delta _A\cdot x+ \mathfrak b_A=0\}
\end{equation}
is satisfied  it thus  suffices to  verify that one point of the set on the left-hand side lies also in the set on the right-hand side.
 This is indeed the case since $o_k\mathfrak{n}_k$ is in the set on the left-hand side and
\begin{equation}
     \delta_A\cdot (o_k \mathfrak{n}_k)+\mathfrak b_A=\lim_{n\to\infty} \underbrace {\delta _k^{(n)} \cdot (o_k^{(n)}\mathfrak{n}_k^{(n)})+ \mathfrak b_k^{(n)}}_{=0},
\end{equation}
where we used that $x=o_k^{(n)}\mathfrak{n}_k^{(n)}$ satisfies 
by assumption $\delta_k^{(n)} \cdot x+\mathfrak b_k^{(n)}=0$.

\underline{Step 3:} The infimum over all strict responses is attained.

We suppose that $(\cR^{(n)})_{n \in \N}$ is a sequence of strict generalized responses satisfying
\begin{equation}
  \lim\limits_{n \to \infty} \int \cL(x, \cR^{(n)}(x)) \, \dd \mu(x) = \inf_{\tilde \cR \in \mathfrak R_d^{\mathrm{strict}}}  \int \cL(x, \tilde\cR(x))\,\dd\mu(x).
\end{equation}
Then we can find a subsequence of responses in $\mathfrak R_{d-1}$ or a subsequence of responses where at least one active area has multiplicity two. 
In the former case the response constructed above is a generalized response of dimension $d-1$ or lower which is strict at dimension $d$. Conversely, in the latter case the construction from above will lead to a generalized response that has at least one active area with multiplicity two which, in turn, implies that $\cR$ is either of dimension strictly smaller than $d$ or is discontinuous. 
\end{proof}

\section{Discontinuous responses are not optimal} \label{sec:opt}
In this section, we show that generalized responses that contain discontinuities are not optimal in the minimization task for a loss function that is continuous in the first argument and strict convex in the second argument. This proves Theorem~\ref{thm:main}  since all continuous generalized responses can be represented by shallow residual ReLU networks.

In the proofs we will make use of the following properties of the  loss functions $\cL$ under consideration.

\begin{lemma} \label{lem:convex}
Let $ \cL \colon \ID \times \R \to [0,\infty) $ be a function satisfying the following assumptions:
\begin{enumerate}
\item[(i)]
(Continuity in the first argument) 
For every $ y \in \R $ it holds that 
$ \ID \ni x \mapsto \cL(x,y) \in \R $ is continuous. 
\item[(ii)] 
(Strict convexity in the second argument) 
For all $ x \in \ID $ it holds that 
$ \R \ni y \mapsto \cL(x,y) \in \R $ 
is strictly convex and attains its minimum.  
\end{enumerate}
Then 
\begin{enumerate}
\item[(I)] 
\label{lem:item:I}
it holds that 
$ \cL \colon \ID\times \R \to [0,\infty) $ 
is continuous, 
\item[(II)] 
\label{lem:item:II}
it holds 
for every compact $ K \subseteq \R $ 
that the function 
$
  \ID \ni x \mapsto \cL( x, \cdot )|_K\in C(K,\R)
$ 
is continuous with respect to the supremum norm, 
\item[(III)] 
\label{lem:item:III}
it holds that there exists a unique 
$
  \mathfrak{m} \colon \ID \to \R
$
which satisfies 
for every $ x \in \ID $ 
that 
\begin{equation}
\textstyle 
  \cL( x, \mathfrak{m}(x) )
  =
  \min_{ y \in \R } \cL(x,y) ,
\end{equation}
and 
\item[(IV)]
\label{lem:item:IV}
it holds that 
$
  \mathfrak{m} \colon \ID \to \R
$
and 
$ 
  \ID \ni x \mapsto \min_{ y \in \R } \cL( x, y ) \in \R
$ 
are continuous.
\end{enumerate}
\end{lemma}

\begin{proof}
Regarding (I): 
The proof can, e.g, be found in \cite[Theorem 10.7]{rockafellar1970convex}.\smallskip

\noindent 
Regarding (II): 
The proof can, e.g., be found in \cite[Theorem 10.8]{rockafellar1970convex}.\smallskip

\noindent 
Regarding (III): 
The existence and uniqueness of $\mathfrak{m} \colon \ID \to \R $ 
is an immediate consequence of property~(ii). 

\noindent 
Regarding (IV): 
We show continuity of $x \mapsto \min_y \cL(x,y)$ and $x \mapsto \mathfrak m(x)$. Let $x' \in \ID$ and choose $y_1, y_2 \in \R$ with $y_1 < \mathfrak m(x') < y_2$.
Now, there exists a neighborhood $U \subseteq \ID$ of $x'$ such that
\begin{equation}
  \cL(x, \mathfrak m(x')) < \cL(x, y_1) \wedge \cL(x,y_2)
\end{equation}
for all $ x \in U $ and 
with the strict convexity of $y \mapsto \cL(x,y)$ we get for all $y \notin [y_1,y_2]$
\begin{equation}
  \cL(x, y) \ge \cL(x, y_1) \wedge \cL(x,y_2).
\end{equation}
Hence, for all $x \in U$ we have $\min_{y \in \R} \cL(x,y)=\min_{y \in [y_1,y_2]} \cL(x,y)$ 
and the continuity of $ x \mapsto \min_y \cL(x,y) $ 
follows from (II). 
Moreover, since $y_1$ and $y_2$ where arbitrary we also have continuity of $x \mapsto \mathfrak m(x)$.
\end{proof}

Clearly, we have that 
\begin{equation}
  \int \cL(x,\mathfrak m(x)) \, \dd \mu(x) = \inf_{f \in C(\ID, \R)} \int \cL(x,f(x)) \, \dd \mu(x)
\end{equation}
so that the minimization task reduces to finding a good approximation 
for $ \mathfrak{m} $ on the support of $ \mu $. 
If $ \mathfrak{m} $ is the response of an ANN $ \IW' \in \cW_d $ for $ d \in \N $, 
then $ \IW' $ minimizes the error 
$ \cW_d \ni \IW \mapsto \err^\cL( \IW ) \in \R $. 
If $ \mathfrak{m} $ is not representable by a network using $ d $ ReLU neurons 
then one expects the minimal error to strictly decrease after adding a ReLU neuron 
to the network structure, i.e., $ \err_{ d + 1 }^{ \cL } < \err_d^{ \cL } $. 
This is true in many situations (see \cref{cor:m}) 
but there exist counterexamples to this intuition 
(see \cref{exa:counter2}).

We note that the set of generalized responses is invariant 
under right applications of affine transformations that are one-to-one.

\begin{lemma}
\label{le:affinetrafo}
Let $ d \in \N_0 $, 
let 
$ \cR \colon \R^{ d_{ \mathrm{in} } } \to \R $ a mapping, 
and 
let $ \varphi \colon \R^{ d_{ \mathrm{in} } } \to \R^{ d_{ \mathrm{in} } } $ 
be an affine mapping that is injective. 
Then the following properties hold for $ \cR $ 
iff the corresponding properties hold for 
$ \cR \circ \varphi $:
\begin{enumerate}
\item[(i)] $\cR\in \mathfrak R_d$
\item[(ii)] $\cR\in \mathfrak R_d^\mathrm{strict}$
\item[(iii)] $\cR$ is simple.
\end{enumerate}
\end{lemma}

\begin{proof}
Suppose that $\cR\in\mathfrak R_d$ 
is represented as in \cref{eq84529} 
with the active areas being for all $ k \in \{ 1, 2, \dots, K \} $
\begin{equation}
  A_k = 
  \{
    x \in \R^{ d_{ \mathrm{in} } } \colon  \vecn_k \cdot x > o_k
  \} 
  .
\end{equation}
We represent the affine function $\varphi$ as  $\varphi(x)=\cA x+b$ with $\cA\in\R^{d_\mathrm{in}\times {d_\mathrm{in}}}$ and $b\in\R^{ d_{ \mathrm{in} } }$. Then obviously
\begin{equation}
  \cR( \varphi( x ) ) 
  = 
  (a \circ \varphi)( x ) 
  +
  \sum_{ k = 1 }^K 
  (
    \delta_k \cdot \varphi( x ) + \mathfrak{b}_k
  )
  \1_{ A_k }( \varphi( x ) ) 
  =
  \tilde{a}(x) +\sum_{k=1}^K\1_{\tilde A_k}(x) (\tilde \delta_k\cdot x+\tilde{\mathfrak b}_k),
\end{equation}
with 
$
  \tilde{a} = a \circ \varphi 
$, 
$
  \tilde{\delta}_k = \cA^\dagger \delta_k 
$, 
$
  \tilde b_k = \mathfrak b_k + \delta_k \cdot b 
$, 
and 
\begin{equation}
  \tilde A_k=\{x\in\R^{ d_{ \mathrm{in} } } \colon  \vecn_k\cdot (\cA x+b) > o_k\}= \{x\in\R^{ d_{ \mathrm{in} } } \colon  (\cA^\dagger \vecn_k)\cdot x   > o_k- \vecn_k\cdot b\}
\end{equation}
where $ ( \cdot )^{ \dagger } $ denotes the transpose of a vector or a matrix. 
Clearly, continuity of the summands is preserved and, hence, one can choose 
the same multiplicities. 
Therefore, $ \cR \circ \varphi $ is again in $ \mathfrak{R}_d $ and 
it is even strict at dimension $ d $ or simple if this is the case for $ \cR $. 
Applying the inverse affine transform $ \varphi^{ - 1 } $ we also obtain equivalence of the properties.
\end{proof}

We are now in the position to prove the main statement of this article, Theorem~\ref{thm:main}. It is an immediate consequence of the following result. We stress that the statement of \cref{prop:minimal} is stronger in the sense that 
it even shows that in many situations the strict generalized responses perform strictly worse 
than the simple responses. 

\begin{prop}\label{prop:minimal}
Suppose that the assumptions of \cref{thm:main} are satisfied. 
Let $d\in\N_0$. Then there exists an optimal network $\IW\in\cW_d$ with 
\begin{equation}
  \mathrm{err}^{ \cL }( \IW )
  =
  \mathrm{err}_d^{ \cL } 
  = 
  \overline{\mathrm{err}}_d^{ \cL } 
  .
\end{equation}
If additionally $ d > 1 $ and $\mathrm{err}^\cL_d<\mathrm{err}^\cL_{d-1}$,
then one has that
\begin{align}\label{eq8246}
\inf_{\cR \in \mathfrak R_d^{\mathrm{strict}}} \int \cL(x,\cR(x))\,\dd \mu(x)>\mathrm{err}_d^\cL.
\end{align}
\end{prop}

\begin{proof} 
We can assume without loss of generality that $ \mu \neq 0 $.
First we verify the assumptions of \cref{prop:genres} 
in order to conclude that there are generalized responses $ \cR $ 
of dimension $ d $ for which
\begin{equation}
  \int \cL(x,\cR(x))\,\dd \mu(x)=\overline{\mathrm{err}}_d^\cL .
\end{equation}
We verify that $ \mu $ is a nice measure: 
In fact, since $ \mu $ has Lebesgue-density $ h $, 
we have $ \mu( H ) = 0 $ for all hyperplanes $ H \subseteq \R^{ d_{ \mathrm{in} } } $. 
Moreover, for every half-space $A$ with $\mu(A),\mu(\overline A^c)>0$ we have that $\partial A$ intersects the interior of the convex hull of $\ID$ so that there exists a point $x \in \partial A$ with $h(x)>0$. Since $ \{x \in \R^{ d_{ \mathrm{in} } } \colon h(x)>0\} $ is an open set, $ \{x \in \partial A: h(x)>0\} $ cannot be covered by finitely many hyperplanes different from $\partial A$. Moreover, since for all $x \in \R^{ d_{ \mathrm{in} } }$ the function $y \mapsto \cL(x,y)$ is strictly convex and attains its minimum we clearly have for fixed $x\in\R^{ d_{ \mathrm{in} } }$ continuity of $y \mapsto \cL(x,y)$ and 
	\begin{align}
	\lim\limits_{|y| \to \infty} \cL(x,y) =\infty.
	\end{align}
	
	We prove the remaining statements via induction over the dimension $d$.
	If $d\le 1$, all generalized responses of dimension $d$ are representable by a neural network and we are done. Now let $d\ge 2$ and suppose that $\cR$ is the best \emph{strict} generalized response at dimension $d$. It suffices to show that one of the following two cases enters: one has 
	\begin{align}\label{eq8362-1}
	\int\cL(x,\cR(x))\,\dd\mu(x) \ge \overline{\mathrm{err}}_{d-1}^\cL
	\end{align} or
	\begin{align}\label{eq8362-2}
	\int\cL(x,\cR(x))\,\dd\mu(x)> \overline{\mathrm{err}}_d^\cL.
	\end{align}
	Indeed, then in the  case that (\ref{eq8362-2}) does not hold we have  as consequence of (\ref{eq8362-1})
\begin{equation}
  \overline{\mathrm{err}}_{d-1}^\cL \le \int\cL(x,\cR(x))\,\dd\mu(x) =\overline{\mathrm{err}}_d^\cL 
\end{equation}
	and the induction hypothesis entails that  $\mathrm{err}_{d-1}^\cL=\overline{\mathrm{err}}_{d-1}^\cL \le \overline{\mathrm{err}}_{d}^\cL\le \mathrm{err}_{d}^\cL\le \mathrm{err}_{d-1}^\cL$ so that $ \mathrm{err}_{d}^\cL= \overline{\mathrm{err}}_{d}^\cL$ and $\mathrm{err}_{d}^\cL= \mathrm{err}_{d-1}^\cL$. 
	Thus, an optimal simple response $\cR$ of dimension $d-1$ (which exists by induction hypothesis) is also optimal when taking the minimum over all generalized responses of dimension $d$ or smaller. 
	Conversely, if~(\ref{eq8362-2}) holds, an optimal generalized response (which exists by  Proposition~\ref{prop:genres}) is simple so that, in particular, $\mathrm{err}_{d}^\cL= \overline{\mathrm{err}}_{d}^\cL$. 
	This shows that there always exists an optimal simple response.
	Moreover, it also follows that in the case where $\mathrm{err}_{d}^\cL< \mathrm{err}_{d-1}^\cL$, either of the properties~(\ref{eq8362-1}) and~(\ref{eq8362-2})  entail property (\ref{eq8246}).
	
	Suppose that $\cR$  is given by
\begin{equation}
  \cR(x) 
  = 
  \mathfrak{a}( x ) 
  + 
  \sum_{ k = 1 }^K 
  \bigl( \delta_k \cdot x + \mathfrak{b}_k \bigr)
  \1_{ A_k }( x ) 
  ,
\end{equation}
with $A_1,\dots,A_K$ being the pairwise different activation areas and $m_1,\dots,m_K$ being the respective multiplicities.
Note that $\cR$ has to be discontinuous, because otherwise $\cR$ 
is of dimension strictly smaller than d and \cref{eq8362-1} holds. 
Therefore, we can assume without loss of generality that $m_K=2$ and
\begin{equation}
  \partial A_K 
  \not \subseteq 
  \{ x \in \R^{ d_{ \mathrm{in} } } \colon \delta_K \cdot x + \mathfrak b_K = 0 \}
\end{equation}
(otherwise we reorder the terms appropriately). 

If $\partial A_K$ does not intersect the interior of $\ID$, then one can replace the term $ \1_{A_K}(x) \bigl(\delta_{K}\cdot x+\mathfrak b_{K}\bigr)$ by
$
  \delta_K \cdot x + \mathfrak{b}_K 
$ 
or 
$ 0 $ 
without changing the error on $ \ID $. 
By doing so the new response has dimension $d-2$ or smaller. Thus, we get that
\begin{equation}
  \int_{ \ID }
  \cL( x, \cR( x ) ) \, 
  \dd\mu( x ) 
  \ge 
  \overline{\mathrm{err}}_{ d - 2 }^{ \cL }
  .
\end{equation}

Now suppose that $ \partial A_K $ intersects the interior of $ \ID $.
We prove that $ \cR $ is not an optimal response in $ \mathfrak{R}_d $ 
by constructing a better response. 
To see this we apply an appropriate affine transformation on the coordinate mapping. 
For an invertible matrix 
$ 
  B \in \R^{ d_{\mathrm{in}} \times d_{\mathrm{in}} } 
$ 
and a vector $ c \in \R^{ d_{ \mathrm{in} } } $ 
we consider the invertible affine mapping
\begin{equation}
  \varphi( x ) = B( x + c ) .
\end{equation}
By \cref{le:affinetrafo} the (strict) generalized responses are invariant under right applications 
of bijective affine transformations so that $ \hat{\cR} = \cR \circ \varphi $ 
is an optimal strict generalized response
for the loss function $ \hat{\cL} $ given by 
$ 
  \hat{\cL}( x, y ) = \cL( x, \varphi^{ - 1 }( y ) ) 
$.

Now we distinguish two cases. 
In line with the notation from before we denote by 
$ \mathfrak{n}_K \in \mathbb S^{d_{ \mathrm{in}-1}} $ and $ o_K \in \R $ 
the unique values for which 
$
  A_K = \{ x \in \R^{ d_{ \mathrm{in} } } \colon \mathfrak{n}_K \cdot x > o_K \} 
$. 
First suppose that $ \delta_K $ and $ \mathfrak{n}_K $ are linearly independent. 
We choose a basis $ \tau_1, \dots, \tau_{d_{\mathrm{in}}} $ of 
$ \R^{ d_{ \mathrm{in} } } $ such that 
$
  \tau_1 \cdot \mathfrak{n}_K = 1 
$, 
$ 
  \tau_1 \perp \delta_K 
$,  
and 
for 
$ 
  \forall \, l \in ( \N \cap [ 1, d_{\mathrm{in}} ] ) \backslash \{ 1 \} \colon
  \tau_l \perp \mathfrak{n}_K 
$. 
This can be achieved by first choosing an arbitrary basis $\tau_2,\ldots,\tau_{d_\mathrm{in}}$ of the space of  vectors being orthogonal to $\mathfrak{n}_K$, secondly choosing a vector $\tau'_1$  that is orthogonal to $\delta_K$ but not to $\mathfrak{n}_K$ which is possible since $\delta_K$ and   $\mathfrak{n}_K$ are linearly independent  and finally letting $\tau_1=_1'/(\tau_1'\cdot \mathfrak{n}_K)$. 

We denote by $ B $ 
the matrix 
$ 
  ( \tau_1, \dots, \tau_{ d_{\mathrm{in}} } ) 
$ 
consisting of the basis vectors and choose $ c \in \R^{ d_{ \mathrm{in} } } $ 
so that
\begin{equation}
  (B^\dagger \mathfrak{n}_K)\cdot c=o_K\text{ \ and \ } (B^\dagger \delta_K)\cdot c=-\mathfrak b_K.
\end{equation}
The latter is feasible since the expression on the left hand side only depends 
on the choice of $ c_1 $ and the expression on the right hand side only 
on the coordinates $ c_2, \dots, c_{d_\mathrm{in}} $.

As is straight-forward to verify the respective response $ \hat{\cR} $ has as $ K $th active area 
$ \hat{A}_K = \{ x \in \R^{ d_{ \mathrm{in} } } \colon x_1 > 0 \} $ 
and on $ \hat{A}_K $ the $ K $th summand in the respective representation 
of $ \hat{\cR} $ is
\begin{equation}
  \delta_K \cdot \varphi( x ) + \mathfrak{b}_K 
  = 
  ( \underbrace{ B^{ \dagger } \delta_K }_{ = \colon \hat{\delta}_K } ) \cdot x 
  .
\end{equation}
Altogether the previous computations show that 
we can assume without loss of generality that the considered strict generalized response 
has as active area $ A_K = \hat{A}_K $ with $ \delta_K $ 
being perpendicular to the first unit vector and $ \mathfrak{b}_K $ being zero.   

We compare the performance of the response $ \cR $ 
with the $ \kappa $-indexed family of generalized responses 
$
  ( \cR^{ \kappa } \colon \kappa \ge 1 ) 
$ 
of dimension $ d $ or smaller given by
\begin{equation}
  \cR^\kappa(x) 
  = 
  \mathfrak a(x)+ \sum_{k=1}^{K-1} \1_{A_k}(x) \bigl(\delta_{k}\cdot x+\mathfrak b_{k}\bigr) +\tilde \cR^\kappa(x),
\end{equation}
where 
\begin{equation}
  \tilde\cR^\kappa(x)= \sfrac 12 (\delta_K\cdot x + \kappa x_1)^{ + }-\sfrac 12 (-\delta_K\cdot x + \kappa x_1)^{ + }.
\end{equation}
Let 
\begin{equation} 
\label{eq74829}
\textstyle 
  \tilde{\cL}(x,y)
  =
  \cL\bigl( x, 
    \mathfrak a(x)
    + 
    \sum_{ k = 1 }^{ K - 1 } 
    ( \delta_k \cdot x + \mathfrak{b}_k )
    \1_{A_k}(x) 
    + y
  \bigr)
  ,
\qquad 
  \tilde{\cR}( x ) 
  =
  \1_{ A_K }( x ) 
  ( \delta_K \cdot x ) ,
\end{equation}
$
  \delta_K' = ( \delta_{ K, 2 }, \dots, \delta_{ K, d_{\mathrm{in}} } )^{ \dagger } 
$ 
and similarly 
$
  x' = ( x_2, \dots, x_{ d_\mathrm{in} } )^{ \dagger } 
$ 
and note that
\begin{equation}
  \tilde{\cR}^{\kappa}( x ) 
  =
  \begin{cases} 
    \tilde\cR(x) ,
  & 
    \text{ if } 
    \kappa|x_1|\ge |\delta_K'\cdot x'|,
  \\
    \text{lin.\ interpol.\ of $\tilde \cR$ between $\check x^{(\kappa)}$ and $\hat x^{(\kappa)}$,}
  & 
    \text{ otherwise,}
  \end{cases}
\end{equation}
where 
\begin{equation}
  \check x^{(\kappa)}
  =
  \begin{pmatrix} -\frac 1\kappa |\delta_K'\cdot x'|\\ x'\end{pmatrix}\text{ \  and \ } \hat x^{(\kappa)}
  =
  \begin{pmatrix} \frac 1\kappa |\delta_K'\cdot x'|\\ x'\end{pmatrix}
  .
\end{equation}
For $ \eps > 0 $ denote by $ \cD_\eps $ the set
\begin{align*}
  \cD_\eps=\Bigl\{x'\in \R^{d_\mathrm{in}-1} \colon \ &\text{the segment $\Bigl[\begin{pmatrix} -\eps \\ x'\end{pmatrix},\begin{pmatrix} \eps \\ x'\end{pmatrix}\Bigr]$ has distance greater or equal to $\eps$ to every $\partial A_k$ ($k\not=K$)}\\
&
  \text{and there exists $x_0\in\R$ so that $(x_0,x')\in \ID$} \Bigr\}
\end{align*}
It is straight-forward to verify that $\cD_\eps$ is closed and due to the compactness of $\ID$ also compact. 
Moreover, 
\begin{equation}
  \bigcup_{\eps>0} \cD_\eps=\underbrace{ \bigl\{x'\in \R^{d_{\mathrm{in}-1}} \colon  \bigl[\exists x_0\in \R: (x_0,x')\in \ID\bigr] \text{ and } \bigl[\forall k\not=K: (0,x') \not \in \partial A_k\bigr] \bigr\}}_{= \colon \cD}.
\end{equation}
We have
	\begin{align}\begin{split}\label{eq84782}
		\int _\ID  \cL(&x,\cR^\kappa(x)) \,h(x)\, dx -\int _\ID   \cL( x ,  \cR(x)) \,h(x)\, dx\\
		&=		\int _\ID  \tilde \cL(x,\tilde \cR^\kappa(x)) \,h(x)\, dx -\int _\ID  \tilde \cL( x , \tilde \cR(x)) \,h(x)\, dx\\
			&=\int _{D_\kappa} \bigl(\tilde \cL(x,\tilde \cR^\kappa(x))- \tilde \cL(x,\tilde \cR(x))\bigr) \,h(x)\, dx,
	\end{split}\end{align}
	where $D_\kappa= \Bigl\{ \begin{pmatrix}x_1\\x'\end{pmatrix} \in \R^{ d_{ \mathrm{in} } } \colon  x'\in \cD, |x_1|\le \sfrac 1{\kappa} |\delta_K'\cdot x'|\Bigr\} $. In dependence on $\eps>0$ we partition $D_\kappa$ 
in two sets
\begin{equation}
  D'_{\kappa,\eps}= \Bigl\{ \begin{pmatrix} x_1\\ x'\end{pmatrix} \in D_\kappa: x'\in \cD_\eps\Bigr\} 
\qquad 
  \text{and} 
\qquad 
  D''_{\kappa,\eps}= \Bigl\{ \begin{pmatrix} x_1\\ x'\end{pmatrix} \in D_\kappa: x'\in \cD\backslash  \cD_\eps\Bigr\}.
\end{equation}
	We note that $\tilde\cL$ is continuous on $([-\eps,\eps]\times \cD_\eps)\times \R^{ d_{ \mathrm{in} } }$ as consequence of the continuity of $\cL$ (see Lemma~\ref{lem:convex}) and the particular choice of $\cD_\eps$. Moreover, for sufficiently large $\kappa$ (depending on $\eps$), $D_{\kappa,\eps}'\subseteq [-\eps,\eps]\times \cD_\eps$. 
	Using that continuous functions are  uniformly continuous on compacts, that $h$ is uniformly bounded and the Lebesgue measure of $D_{\kappa}$ is of order  $\cO(1/\kappa)$ we conclude that in terms of $\Upsilon_{x'} :=|\delta_K'\cdot x'|$
\begin{equation}
\begin{split}
\label{eq578}
&
  \int_{ D_{ \kappa, \eps }' } 
  \bigl(\tilde \cL(x,\tilde \cR^\kappa(x))- \tilde \cL(x,\tilde \cR(x))\bigr) \,h(x)\, dx 
\\
&
  = \int_{D_{\kappa,\eps}'} \Bigl(\tilde \cL\Bigl(\begin{pmatrix}0\\ x'\end{pmatrix},\tilde \cR^\kappa(x)\Bigr)- \tilde \cL\Bigl(\begin{pmatrix}0\\ x'\end{pmatrix},\tilde \cR(x)\Bigr)\Bigr) \,h\Bigl(\begin{pmatrix}0\\ x'\end{pmatrix}\Bigr)\, dx+ o\bigl(\sfrac 1\kappa\bigr)\\
				&=\int _{\cD_\eps} \Bigl(\int_{-\frac 1\kappa\Upsilon_{x'}}^{\frac 1\kappa\Upsilon_{x'}} \Bigl(\tilde \cL\Bigl(\begin{pmatrix}0\\ x'\end{pmatrix},\tilde \cR^\kappa\Bigl(\begin{pmatrix}x_1\\ x'\end{pmatrix}\Bigr)\Bigr)- \tilde \cL\Bigl(\begin{pmatrix}0\\ x'\end{pmatrix},\tilde \cR\Bigl(\begin{pmatrix}x_1\\ x'\end{pmatrix}\Bigr)\Bigr)\Bigr)\,dx_1 \Bigr) \,h\Bigl(\begin{pmatrix}0\\ x'\end{pmatrix}\Bigr)\, dx'+ o\bigl(\sfrac 1\kappa\bigr)
\end{split}
\end{equation}
as $\kappa\to\infty$. Moreover, for every $x'\in \cD_\eps$ one has 
\begin{equation}
\begin{split}
		\int_{-\frac 1\kappa \Upsilon_{x'}}^{\frac 1\kappa\Upsilon_{x'}}& \Bigl(\tilde \cL\Bigl(\begin{pmatrix}0\\ x'\end{pmatrix},\tilde \cR^\kappa\Bigl(\begin{pmatrix}x_1\\ x'\end{pmatrix}\Bigr)\Bigr)- \tilde \cL\Bigl(\begin{pmatrix}0\\ x'\end{pmatrix},\tilde \cR\Bigl(\begin{pmatrix}x_1\\ x'\end{pmatrix}\Bigr)\Bigr)\Bigr)\, dx_1\\
		&=\frac 1\kappa \int_{- \Upsilon_{ x'}}^{\Upsilon_{ x'}} \Bigl(\tilde \cL\Bigl(\begin{pmatrix}0\\ x'\end{pmatrix},\tilde \cR^1\Bigl(\begin{pmatrix}x_1\\ x'\end{pmatrix}\Bigr)\Bigr)- \tilde \cL\Bigl(\begin{pmatrix}0\\ x'\end{pmatrix},\tilde \cR\Bigl(\begin{pmatrix}x_1\\ x'\end{pmatrix}\Bigr)\Bigr)\Bigr)\, dx_1 = \colon \frac 1\kappa\, Q (x'). 
\end{split}
\end{equation}
We represent $Q(x')$  in terms of the measure $ \nu_{x'}=\mathrm{Leb} |_{[- \Upsilon_{ x'}, \Upsilon _{ x'}]}$, the strictly convex function
\begin{equation}
  \xi_{x'} \colon [-\Upsilon_{ x'}, \Upsilon_{x'}] \to[0,\infty), \ x_1\mapsto \tilde \cL\Bigl(\begin{pmatrix}0
\\ 
  x'\end{pmatrix},\tilde \cR^1\Bigl(\begin{pmatrix}x_1\\ x'\end{pmatrix}\Bigr)\Bigr)
\end{equation}
 and its secant $ \bar \xi_{x'} \colon [-\Upsilon_{ x'}, \Upsilon_{x'}]\to[0,\infty) $ 
 that equals $ \xi_{x'} $ 
 in the boundary points and is linear in between. 
 One has
\begin{equation}
  Q( x' ) 
  =
  \int \bigl(\xi_{x'}(z) - \bar  \xi_{x'}(z)\bigr)\, d \nu_{x'} \le 0
\end{equation}
with strict inequality in the case where $ \Upsilon_{ x' } > 0 $ 
(due to strict convexity). 
Consequently, we get with~\cref{eq578} that
\begin{align}
\label{eq9357}
  \lim_{ \kappa \to \infty } \kappa 
  \int_{ D_{ \kappa, \eps }' } 
  \bigl(\tilde \cL(x,\tilde \cR^\kappa(x))- \tilde \cL(x,\tilde \cR(x))\bigr) 
  \, h(x) \, dx
  = 
  \int_{\cD_\eps}    Q(x') \, h\Bigl(\begin{pmatrix}0\\ x'\end{pmatrix}\Bigr)\, dx' 
  .
\end{align}

To analyze the contribution of the integrals on $D_{\kappa,\eps}''$ we note that  by uniform boundedness of $\tilde \cL(x,\tilde \cR^\kappa(x))- \tilde \cL(x,\tilde \cR(x))$ over all $x\in \ID$ and $\kappa\ge1$ one has existence of a constant $C$ not depending on $\kappa$ and $\eps$ such that
\begin{equation}
  \Bigl|\int _{D_{\kappa,\eps}''} \bigl(\tilde \cL(x,\tilde \cR^\kappa(x))- \tilde \cL(x,\tilde \cR(x))\bigr) \,h(x)\, dx\Bigr| \le C \, |\cD\backslash \cD_\eps|\, \sfrac 1\kappa,
\end{equation}
where $|\cD\backslash \cD_\eps|$ is the $(d_{\mathrm{in}}-1)$-dimensional Hausdorff measure of the set $\cD \backslash \cD_\eps$.
By choosing $\eps>0$ arbitrarily small one can make $|\cD\backslash \cD_\eps|$ arbitrarily small and with a diagonalization argument we obtain with~(\ref{eq84782}) and~(\ref{eq9357})  that 
\begin{equation}
  \lim_{\kappa\to\infty} \kappa	\int _{\ID} \bigl( \cL(x, \cR^\kappa(x))-  \cL(x, \cR(x))\bigr) \,h(x)\, dx 
  = 
  \int_{\cD}    Q(x') \, h\Bigl(\begin{pmatrix}0\\ x'\end{pmatrix}\Bigr)\, dx'
  .
\end{equation}

Now there exists $ x' \in \cD $ 
with $ h( x' ) > 0 $ 
and by continuity of $ h $ 
we can choose $ x' $ such that, 
additionally, 
$ 
  \Upsilon_{x'} = |\delta_K'\cdot x'| > 0 
$. 
By continuity we thus get that
\begin{equation}
  \int_{\cD}    Q(x') \, h\Bigl(\begin{pmatrix}0\\ x'\end{pmatrix}\Bigr)\, dx'<0
\end{equation}
and, consequently, there exists $\kappa>0$ such that the generalized response $\cR^\kappa$ of dimension $d$ or smaller has a strictly smaller error than $\cR$. Hence, it has to be simple.

It remains to treat the case where $\delta_K$ and $\mathfrak{n}_K$ are linearly dependent. In that case we choose $\alpha\in\R$ with $\delta_K=\alpha \mathfrak{n}_K$, we extend $\tau_1=\mathfrak{n}_K$ to an orthonormal basis $(\tau_1,\dots,\tau_{d_\mathrm{in}})$ of $\R^{ d_{ \mathrm{in} } }$ and denote by $B$ the matrix formed by the vectors $\tau_1,\dots,\tau_{d_\mathrm{in}}$. Moreover, choose 
$c=(o_K,0,\dots,0)^\dagger$ and set $\varphi(x)=B(x+c)$. Then the response $\hat \cR=\cR\circ \varphi$ has as $K$th activation area
$\hat A_K= \{x\in \R^{ d_{ \mathrm{in} } } \colon  x_1>0\} $ and on $\hat A_K$ the $K$th summand in the respective representation of $\hat \cR$ is
\begin{equation}
  \delta_K\cdot \varphi(x)+\mathfrak b_K= \alpha (x_1+o_K)+\mathfrak b_K= \alpha x_1+\hat {\mathfrak b}_K 
  ,
\end{equation}
where $\hat {\mathfrak b}_K=\alpha o_K+\mathfrak b_K\not=0$ 
since otherwise we would have that
\begin{equation}
  \partial A_K 
  \subseteq \{ x \in \R^{ d_{ \mathrm{in} } } \colon \delta_K \cdot x + \mathfrak b_K = 0 \} .
\end{equation}
We showed that in the remaining case we can assume without loss of generality that $A_K=\hat A_K$, $\delta_K=(\alpha,0,\dots,0)^\dagger$ for an $\alpha\in\R$  and $\mathfrak b_K\neq 0$.

In analogy to above we compare the response $ \cR $ with the $ \kappa $-indexed 
family of responses $ ( \cR^{ \kappa } \colon \kappa \ge 1 ) $ given by
\begin{equation}
  \cR^{ \kappa }(x) 
  = 
  \mathfrak{a}( x ) 
  + 
  \sum_{ k = 1 }^{ K - 1 } 
  \bigl( 
    \delta_k \cdot x + \mathfrak{b}_k 
  \bigr) 
  \1_{ A_k }( x ) 
  +
  \tilde{\cR}^\kappa( x ) 
  ,
\end{equation}
where 
\begin{equation}
  \tilde{\cR}^{ \kappa }( x ) 
  = \sfrac 12 (\alpha + \mathfrak b_K \kappa) \bigl( x_1 + \sfrac 1\kappa )^{ + }
  +
  \sfrac 12 (\alpha - \mathfrak b_K \kappa) \bigl(x_1-\sfrac1\kappa\bigr)^{ + } 
  .
\end{equation}
We use $ \tilde{\cL} $ and $ \tilde{\cR} $ as before, 
see~\cref{eq74829}, 
and note that $ \tilde{\cR}^{\kappa} $ agrees 
with $ \tilde{\cR} $ 
for all $ x \in \R^{ d_{ \mathrm{in} } } $ 
with $ |x_1| \ge \frac 1 \kappa $.
In analogy to above, we conclude that
\begin{equation}
\begin{split}
  \int _{\ID}& \bigl( \cL(x, \cR^\kappa(x))-  \cL(x, \cR(x))\bigr) \,h(x)\, dx\\
& 
  = \int _{D_{\kappa}} \Bigl(\tilde \cL\Bigl(\begin{pmatrix}0\\ x'\end{pmatrix},\tilde \cR^\kappa(x)\Bigr)- \tilde \cL\Bigl(\begin{pmatrix}0\\ x'\end{pmatrix},\tilde \cR(x)\Bigr)\Bigr) \,h\Bigl(\begin{pmatrix}0\\ x'\end{pmatrix}\Bigr)\, dx+ o\bigl(\sfrac 1\kappa\bigr) 
  ,
\end{split}
\end{equation}
where 
$ 
  D_\kappa = [-\frac 1\kappa,\frac 1\kappa] \times \cD 
$.  
As above we split the domain of integration into the two sets 
$ D_{ \kappa, \eps }' $ and $ D_{ \kappa, \eps }'' $. 
Now in terms of 
\begin{equation}
  \xi_{x'} \colon [-1,1] \to[0,\infty), 
  \qquad
  x_1 \mapsto 
  \tilde{\cL}\biggl(
    \begin{pmatrix}
      0 \\ x'
    \end{pmatrix}
    ,  
    \sfrac 12( x_1 + 1 ) \mathfrak{b}_K 
  \biggr)
\end{equation}
we get by using the uniform continuity 
of $\tilde \cL$ on $D_{\kappa,\eps}'$ and 
the fact that $|D_{\kappa,\eps}|=\cO(\frac 1\kappa)$ as $\kappa\to\infty$ that
\begin{equation}
\begin{split}
  \int _{D_{\kappa,\eps}'}& \Bigl(\tilde \cL\Bigl(\begin{pmatrix}0\\ x'\end{pmatrix},\tilde \cR^\kappa(x)\Bigr)- \tilde \cL\Bigl(\begin{pmatrix}0\\ x'\end{pmatrix},\tilde \cR(x)\Bigr)\Bigr) \,h\Bigl(\begin{pmatrix}0\\ x'\end{pmatrix}\Bigr)\Bigr)\, dx
\\
&
  =\frac 1\kappa \int_{\cD_\eps}\int _{-1}^{1} \xi_{x'}(x_1)\,dx_1 -\bigl(\xi_{x'}(-1)+\xi_{x'}(1)\bigr )\,h\Bigl(\begin{pmatrix}0\\ x'\end{pmatrix}\Bigr)\, dx'+ o\bigl(\sfrac 1\kappa\bigr)
  .
\end{split}
\end{equation}
By strict convexity of $\xi_{x'}$, 
we get that $Q(x'):=\int _{-1}^{1} \xi_{x'}(x_1)\,dx_1 -\bigl(\xi_{x'}(-1)+\xi_{x'}(1)\bigr )<0$. 
With the same arguments as in the first case one obtains that
\begin{equation}
  \lim_{\kappa\to\infty} \kappa	\int _{\ID} \bigl( \cL(x, \cR^\kappa(x))-  \cL(x, \cR(x))\bigr) \,h(x)\, dx 
  = 
  \int_{\cD}    Q(x') \, h\Bigl(\begin{pmatrix}0\\ x'\end{pmatrix}\Bigr)\, dx' < 0
\end{equation}
so that there exists a response $ \cR^{ \kappa } $ 
with strictly smaller error than $ \cR $ and the proof is finished.
\end{proof}

\begin{exa} 
\label{exa:counter}
If there exists a hyperplane $ H $ with $ h(x) = 0 $ 
for all $ x \in H $ such that $ H $ intersects 
the convex hull of $ \supp( \mu ) $ the conclusion of \cref{thm:main} 
is in general not true. 
Consider a continuous function $ f \colon \R^2 \to \R $ 
that satisfies $ f(x) = 1 $ for all $ x \in B( (0,1), 1 ) $ 
and $ f(x) = 0 $ for all $ x \in B( (1,-1), 1 ) \cup B( (-1,-1), 1 ) $. 
Now, let $ \cL( x, y ) = ( f(x) - y )^2 $ and $ \mu $ 
be the measure on $ \R^2 $ with continuous Lebesgue density
\begin{equation}
  h(x) = 
  \1_{ B( (0,1), 1 )}( x ) \, | x - (0,1) | 
  + 
  \1_{ B( (1,-1), 1 )}( x ) \, | x - (1,-1) | 
  +
  \1_{ B( (-1,-1), 1 )}( x ) \, | x + (1,1) | 
  .
\end{equation}
Then, we have $ h(x) \equiv 0 $ on $ \R \times \{ 0 \} $. 
Note that
\begin{equation}
  \cR(x) = 
    \begin{cases}
      0, 
    & 
      \text{ if } x_2\leq0 
    \\
      1, 
    & 
      \text{ if } x_2 > 0
    \end{cases} 
\end{equation}
is a strict generalized response of dimension $ 2 $ 
with $ \int \cL( x, \cR( x ) ) \, \dd \mu( x ) = 0 $. 
Conversely, there does not exist a network $ \IW \in \cW_d $ 
(for arbitrary $ d \in \N $) 
with 
$
  \int \cL( x, \mathfrak N^{ \IW }( x ) ) \, \dd \mu ( x ) = 0 
$. 
In particular, assume that there exist 
$ d \in \N $ and $ \IW \in \cW_d $ 
with $ \err^{ \cL }( \IW ) = 0 $. 
Then, every breakline $ H \neq \R \times \{ 0 \} $ 
that intersects the interior of $ D $ contains an uncountable set 
of points $ x \in H $ with $ h(x) > 0 $ and for all such points 
the function $ f $ is constant in a neighborhood of $ x $. 
Therefore, the collective response of the neurons with breakline $ H $ has 
to be constant and we can without loss of generality assume that all ReLU neurons 
have the breakline $ H = \R \times \{ 0 \} $. 
Moreover, since $ \err^{ \cL }( \IW ) = 0 $ 
it holds that $ \mathfrak{N}^{ \IW }( x ) = 0 $ 
for all $ x = ( x_1, x_2 ) \in \R^2 $ with $ x_2 < 0 $ and $ \mathfrak{N}^{ \IW }( x ) = 1 $ 
for all $ x = ( x_1, x_2 ) \in \R^2 $ with $ x_2 > 0 $. 
This contradicts the continuity of $ \mathfrak{N}^{ \IW } $.

Thus, there does not exist a global minimum in the loss landscape 
$ \cW_d \ni \IW \to \err^{ \cL }( \IW ) $ (for arbitrary $ d \in \N $) 
and, in order to solve the minimization task iteratively, 
the sequence of networks returned by a gradient based algorithm 
have unbounded parameters.

For a thorough investigation of the non-existence of global minima 
in the approximation of discontinuous target functions $ f $, 
see \cite{gallonjentzenlindner2022blowup}. 
\begin{figure}
  \includegraphics[width=.3\linewidth]{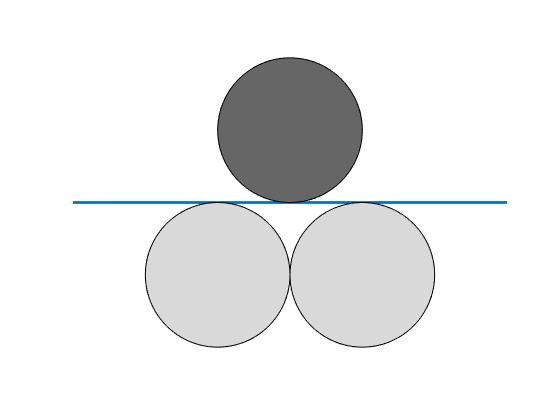}
  \caption{Visualization of the minimization task in \cref{exa:counter}. 
  There exists a generalized response of dimension $ 2 $ but 
  no neural network $\IW \in \cW_d$ ($d \in \N$) attaining zero error.}
\end{figure}
\end{exa}

Next, we show that in many situations if the class of network responses is not able to produce the function $x \mapsto \mathbf m(x)$ defined in Lemma~\ref{lem:convex} the minimal error strictly decreases after adding a ReLU neuron to the network structure.

\begin{prop} 
\label{cor:m} 
Let $ d \in \N_0 $, assume that $ \ID = \supp( \mu ) $ is a compact set, 
assume there exists $ \IW \in \cW_d $ with $ \err^{ \cL }( \IW ) = \err_d^{ \cL } < \infty $,
assume for every $ x \in \ID $ that the function $ \cL( x, \cdot ) $ is convex, 
\begin{enumerate}
\item[(i)] 
assume for every compact $ K \subseteq \R $ there exists $ L \in \R $ such that 
for all $ x \in \ID $, $ y, y' \in K $ that
\begin{equation}
  | \cL( x, y ) - \cL( x, y' ) | \le L | y - y' | ,
\end{equation}
\item[(ii)] 
assume for every affine function $ \varphi \colon \R^{ d_{ \mathrm{in} } } \to \R $ 
that the set
\begin{equation}
  \bigl\{ 
    x \in \ID \colon 
    \cL\text{ is not $y$-differentiable in }(x,\varphi(x))
  \bigr\}
\end{equation}
is a Lebesgue nullset,
\end{enumerate}
and assume that there exist no neural network 
$ \IW \in \cW_d $ satisfying for $ \mu $-almost all $ x \in \ID $ that
\begin{equation}
\textstyle 
  \cL( x, \mathfrak{N}^{ \IW }( x ) ) 
  =
  \inf_{ y \in \R } \cL( x, y ) 
  .
\end{equation}
Then 
\begin{equation}
  \err_d^{ \cL } > \err_{ d + 1 }^{ \cL } 
  .
\end{equation}
\end{prop}

\begin{rem}
We compare the assumptions of \cref{cor:m} 
with those of \cref{thm:main}. 
In \cref{cor:m} we explicitly assume 
the existence of an optimal network $ \IW \in \cW_d $ and 
relax the continuity assumption on $ \cL $ in the first component 
and the strict convexity assumption on $ \cL $ in the second component. 
On the other hand, we introduce assumptions on the smoothness of $ \cL $ in the second component. 
Under the assumptions of \cref{thm:main}, condition (i) of \cref{cor:m} is satisfied 
(cf.\ \cite[Thm. 10.6]{rockafellar1970convex}) and we can apply the latter proposition 
if additionally condition (ii) is satisfied. In that case, the statement of \cref{cor:m} 
can be rewritten as follows: 
if $ \err_{d}^{ \cL } \le \err_{ d + 1 }^{ \cL } $, then there exists a neural network 
$ \IW \in \cW_d $ with $ \mathfrak{N}^{ \IW }( x ) = \mathfrak{m}( x ) $ for all $ x \in \ID $.
\end{rem}

\begin{proof}[Proof of \cref{cor:m}]
Let $ \IW \in \cW_d $ be a network with $ \err^{ \cL }( \IW ) = \err_d^{ \cL } $.
For $ \Delta, o \in \R $, $ \vecn \in \mathbb S^{d_{ \mathrm{in}-1}} $ consider the function
\begin{equation}
  N(\Delta, \vecn, o)(x) = \mathfrak N^{\IW}(x) + \Delta (\vecn\cdot x-o)^{ + }
  .
\end{equation}
If $ \err_d^{ \cL } \le \err_{ d + 1 }^{ \cL } $ then 
we have for all $ \Delta, o \in \R $, $ \vecn \in \mathbb S^{d_{ \mathrm{in}-1}} $ that
\begin{equation}
  \int_{\ID} \cL(x, \mathfrak N^{\IW}(x))\, h(x) \, dx 
  \le 
  \int_{\ID} \cL(x, N(\Delta, \vecn, b)(x)) \,h(x) \, dx
\end{equation}
and taking the derivative with respect to $ \Delta $ at $ \Delta = 0 $ yields
\begin{equation} 
  \int_{\ID} 
  \Bigl( 
    \frac{ \partial }{ dy }
    \cL( x, \mathfrak{N}^{ \IW }( x ) ) 
  \Bigr) 
  ( \vecn \cdot x - o )^{ + } \, h(x) \, dx 
  = 0, 
\end{equation}
where 
$ 
  x \mapsto \frac{ \partial }{ dy } \cL( x, \mathfrak{N}^{ \IW }( x ) )
$ 
is uniformly bounded and 
well-defined outside a Lebesgue nullset. 
Indeed, since $ \mathfrak{N}^{ \IW } $ is piecewise affine 
there exists a finite number of affine functions 
$ 
  \varphi_1, \dots, \varphi_m 
$ 
such that the set of points $ x \in \ID $ for which $ \cL $ is not $ y $-differential 
in $ ( x, \mathfrak{N}^{ \IW }( x ) ) $ 
is contained in 
\begin{equation}
  \bigcup_{ i = 1 }^m
  \{ 
    x \in \ID \colon \cL \text{ is not $ y $-differentiable in } (x,\varphi_i(x)) 
  \} 
  ,
\end{equation}
which by (ii) is a nullset. Moreover, the boundedness of the derivative follows 
from the Lipschitz continuity of $ \cL $ in the second argument, see (i). 
We let 
\begin{equation}
  \tilde h(x) := \Bigl( \frac{\partial}{\partial y}\cL(x,\mathfrak N^{\IW}(x)) \Bigr) h(x)
\end{equation}
and note that the space $ \cH $ of all continuous functions 
$ g \colon \ID \to \IC $ satisfying
\begin{equation}
  \int_\ID g(x) \,\tilde h(x) \, dx=0
\end{equation}
is linear and closed under convergence in $ C( \ID, \IC ) $ 
(endowed with supremum norm). 
We showed that $ \cH $ contains all functions 
of the form $ x \mapsto ( \vecn \cdot x - o )^{ + } $ and 
it is standard to deduce that $ \cH $ contains 
all polynomials and, using the Stone-Weierstrass theorem, 
thus all continuous functions. By the Riesz--Markov--Kakutani representation theorem, 
the measure $ \tilde{h}( x ) \, dx $ is the zero-measure and 
$ \tilde{h} $ is zero except for $ \mu $-nullsets. 
Note that $ h > 0 $, $ \mu $-almost everywhere, and 
hence 
$
  \frac \partial{ \partial y } \cL( x,\mathfrak N^{\IW}(x))=0
$, 
$ \mu $-almost everywhere. 
Using the convexity of $y \mapsto \cL(x,y)$ we get $\cL(x,\mathfrak N^{\IW}(x))=\inf_{y \in \R} \cL(x,y)$, $\mu$-almost everywhere.
	\end{proof}

In the next example, we show that the conclusion of Proposition~\ref{cor:m} is in general false if condition (ii) is not satisfied. We note that the loss function $\cL$ in the example is not strictly convex but the statements of Lemma~\ref{lem:convex} still hold in this case.

\begin{exa} 
\label{exa:counter2}
Consider the following regression problem. 
Let $d_{\mathrm{in}}=1$, $\ell \ge 1$, $\mu = \mathrm{Leb} |_{[-\ell-1,1+\ell]}$ 
be the Lebesgue measure on the interval $[-\ell-1,1+\ell]$ and $\cL(x,y)=|y-f(x)|$ where
\begin{equation}
  f(x)= \left\{ \begin{matrix} 1-|x|, & \text{ if } |x|\le 1 \\ 0, & \text{ if } |x|>1 \end{matrix} \right. .
\end{equation}
Note that $\lambda(\{x \in [-\ell-1, \ell+1]: \cL \text{ is not $y$-differentiable in } (x,0)\})=2 \ell>0$ and 
\begin{equation}
  f(x) = (x+1)^{ + } -2(x)^{ + } +(x-1)^{ + }
\end{equation}
so that $ f $ is the response of a network using three ReLU neurons and $ \err_3^{ \cL } = 0 $. 
We show that for $ \ell \ge 13 $ we have
\begin{equation}
  \err_0^\cL=\err_1^{\cL} = \err_2^{\cL} = \int_{-\ell-1}^{\ell+1} |f(x)| \, dx = 1 ,
\end{equation}
i.e. the best regression function in the set of response functions for networks having $ 2 $ ReLU neurons is the zero function although 
$
  \inf_{ y \in \R } \cL( x, y ) = f(x) 
$ 
is not the zero function. 
This shows that the conclusion of \cref{cor:m} is in general false if Assumption~(ii) 
is not satisfied.

Denote by $ E \colon \cW_2 \times \II \to \R $ the function
\begin{equation}
  ( \IW, [s,t] ) \mapsto E( \IW, [s,t] ) 
  = 
  \int_s^t 
  | \mathfrak{N}^{ \IW }( x ) - f(x) | - | f(x) | \, dx ,
\end{equation}
where $ \II $ is the set of all closed intervals that are subsets of $ [ - \ell - 1, \ell + 1 ] $. 
We show that for $ \ell \ge 13 $ we have $ E( \IW, [-\ell-1, 1 + \ell ] ) \ge 0 $ 
for all networks $ \IW \in \cW_2 $.

Let $ \IW \in \cW_2 $. 
Then $ \mathfrak{N}^{ \IW } $ 
is continuous and satisfies
\begin{equation}
  \mathfrak{N}^{ \IW }( x ) 
  = 
  \left\{ \begin{matrix} \delta_1  x + \mathfrak b_1, & \text{ if } x \le o_1 \\ \delta_2  x + \mathfrak b_2,& \text{ if } o_1 < x < o_2 \\ \delta_3  x + \mathfrak b_3,& \text{ if } x > o_2 \end{matrix} \right. 
  ,
\end{equation}
for $ \delta_1, \delta_2, \delta_3, \mathfrak{b}_1, \mathfrak{b}_2, \mathfrak{b}_3, o_1, o_2 \in \R $ 
with $ o_1 \le o_2 $.
\begin{figure}
\includegraphics[width=.35\linewidth]{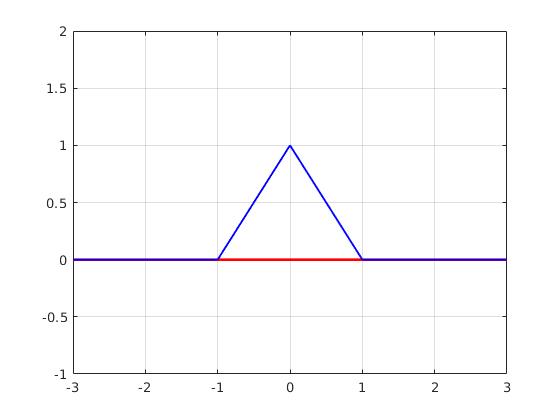}
\caption{Visualization of the minimization task in Example~\ref{exa:counter2}. There exists a network $\IW \in \cW_3$ with $\mathfrak N^\IW=f$ ({\blue blue}). 
However, the realization function attaining minimal error 
in the class $\cW_2 $, $ \cW_1 $, and $ \cW_0 $ 
is $ \mathfrak{N}^{ \IW } = 0 $ ({\red red}).}
\end{figure}

First, we assume that $o_1 \le 0 \le o_2$ and $0\le \mathfrak N^{\IW}(0)$. 
We fix $\delta_2, \mathfrak b_2 \in \R$ and derive a lower bound for $E(\IW,[0,\ell+1])$ over all feasible choices of $o_2, \delta_3$ and $\mathfrak b_3$.
We start with the case $\delta_2<0$. 
Note that the choice $o_2=\mathfrak N^{\IW}(0)/|\delta_2|$, $\delta_3=\mathfrak b_3=0$ yields a better result than all networks with $o_2\ge \mathfrak N^{\IW}(0)/|\delta_2|$. Therefore we can restrict the optimization task to networks satisfying 
$
  o_2 \le \mathfrak N^{\IW}(0)/|\delta_2|.
$
For $\mathfrak N^{\IW}(0)/|\delta_2|\le 1$ we show that, indeed, the optimal choice for the approximation of $f$ on the right-hand side of the $y$-axis is $o_2=\mathfrak N^{\IW}(0)/|\delta_2|$, $\delta_3=\mathfrak b_3=0$. If $o_2<\mathfrak N^{\IW}(0)/|\delta_2|$ and $\mathfrak N^{\IW}(1)\le 0$ then 
\begin{equation}
  E(\IW,[o_2,a])
  \ge - \frac 12 \mathfrak N^{\IW}(o_2)(a-o_2) 
  \qquad 
  \text{and} 
  \qquad 
  E(\IW,[a,\ell+1])\ge \frac 12 \mathfrak N^{\IW}(o_2)(a-o_2),
\end{equation}
where $a \in [0,1]$ with $\mathfrak N^{\IW}(a)=0$. Conversely, if $\mathfrak N^{\IW} (1) \ge 0$ then $E(\IW,[o_2,1])\ge - \mathfrak N^{\IW}(o_2)(1-o_2)\ge - \mathfrak N^{\IW}(o_2)$. We consider two cases. If $\mathfrak N^{\IW}(\ell+1)\ge 0$ then $\delta_3\ge - \mathfrak N^{\IW}(o_2)/\ell$. Thus, $E(\IW,[1,\ell+1])\ge \frac 12 \ell \mathfrak N^{\IW}(1) \ge \frac {\ell-1}{2} \mathfrak N^{\IW}(o_2)$ so that for $\ell \ge 3$ we get $E(\IW, [o_2,\ell+1])\ge 0$. If $\mathfrak N^{\IW}(\ell+1)\le  0$ then $\delta_3 \le - \mathfrak N^{\IW}(o_2)/(\ell+1)$. In that case, $E(\IW,[1,\ell+1])$ corresponds to the area of two triangles with slope $|\delta_3|$ and baselines that add to $\ell$. This is minimized by two congruent triangles so that
\begin{equation}
  E(\IW, [1,\ell+1]) \ge  \Bigl(\frac{\ell}{2}\Bigr)^2 |\delta_3| \ge \frac{\ell-1}{4}\mathfrak N^{\IW}(o_2)
  .
\end{equation}
Thus, 
for $\ell\ge 5$ 
we get $E(\IW, [o_2,\ell+1])\ge 0$ and 
the optimal choice is $o_2 = \mathfrak N^{\IW}(0)/|\delta_2|$, $\delta_3=\mathfrak b_3=0$.

It remains to consider the case $\mathfrak{N}^{ \IW }( 0 ) / | \delta_2 | \ge 1 $. In this case the choice $\delta_3 >0$ is clearly suboptimal and for $o_2 \le 1$ 
we get $ E( \IW, [0,o_2] ) \ge - ( \mathfrak{N}^{ \IW }( 0 ) - \frac 12 | \delta_2 | ) $. 
The above calculations show that 
\begin{equation}
  E( \IW, [o_2,\ell+1] ) 
  \ge \frac{ \ell - 5 }{ 4 } \mathfrak{N}^{ \IW }( o_2 ) 
  \ge \frac{ \ell - 5 }{ 4 }( \mathfrak{N}^{ \IW }( 0 ) - | \delta_2 | )
\end{equation}
so that for $\ell \ge 7$ we have 
$
  E(\IW,[0,\ell+1]) \ge - \frac 12 \mathfrak{N}^{\IW}(0) 
$.  
If $\mathfrak N^{\IW}(0)/|\delta_2| \ge \ell+1$ the choice $o_2>1$ is suboptimal and 
in the case $\mathfrak N^{\IW}(0)/|\delta_2| \le \ell+1$ where $o_2\le 1$ is not the optimal choice 
it is easy to see that actually $o_2=\mathfrak N^{\IW}(0)/|\delta_2|$, $\delta_3=\mathfrak b_3=0$ 
is the best choice. In that case
\begin{equation}
  E(\IW,[0,1]) \ge -(\mathfrak N^{\IW}(0)-\frac 12 |\delta_2|) \quad \text{ and } \quad E(\IW,[1,\ell+1]) \ge \frac 12 (\mathfrak N^{\IW}(0)-|\delta_2|).
\end{equation}
In conclusion, in all of the above cases 
we get for $\ell \ge 7$ that $E(\IW, [0, \ell+1])\ge -\frac 12 \mathfrak N^{\IW}(0)$.

Next, we consider the case $\delta_2 \ge 0$. 
Set $ a := \inf\{x \ge 0 \colon \delta_2x+\mathfrak b_2\ge 1-x\} $. 
Choosing $o_2 > a$ or $\delta_3 \ge 0$ is clearly suboptimal. 
For $0 \le o_2 \le a$ and $\delta_3<0$ note that $E(\IW, [0, 1]) \ge -\mathfrak N^{\IW}(o_2)$. 
Analogously to the case $\delta_2 < 0$ we get
\begin{equation}
  E(\IW, [1,\ell+1]) \ge \frac{\ell-1}{4} \mathfrak N^{\IW}(o_2)
\end{equation}
so that for $\ell \ge 7$
\begin{equation}
  E(\IW, [0, \ell+1]) \ge \frac 12 \mathfrak N^{\IW}(o_2) \ge \frac 12 \mathfrak N^{\IW}(0)
  .
\end{equation}
Now, if $\mathfrak N^{\IW}(0)\le 0$ and $\ell\ge 5$ then 
the above calculations imply that $E(\IW,[-\ell-1,\ell+1])\ge 0$.
Using the symmetry of the problem we showed that 
for all networks $\IW\in \cW_2$ satisfying $o_1\le0\le o_2$ we have that
\begin{equation}
  \int_{-\ell-1}^{\ell+1} \cL(x, \mathfrak N^{\IW}(x)) \, \dd x \ge \int_{-\ell-1}^{\ell+1} \cL(x, 0) \, \dd x
  .
\end{equation}

Using again the symmetry, we are therefore left with considering 
the case $ 0 < o_1 \le o_2 $. 
We can clearly focus on the case 
$ 
  o_1 \le 1 
$ 
and 
$ 
  \mathfrak{N}^{ \IW }( 0 ) \ge 0
$.
Note that one can use the above arguments in order to show that 
\begin{equation}
  E(\IW,[o_1,\ell+1]) 
  \ge 
  \min\bigl( 0, - \frac 12 \mathfrak N^{\IW}( o_1 ) \bigr)
  .
\end{equation}
If $\delta_1<0$ we thus get
\begin{equation}
  E( \IW, [0, \ell + 1 ] ) 
  \ge - \frac 32 \mathfrak{N}^{\IW}(0)
  .
\end{equation}
On the other hand, 
$ 
  E( \IW, [ - \ell - 1, 0 ] )
  \ge ( \ell - 1 ) \mathfrak{N}^{ \IW }( 0 ) 
$ 
so that, 
for $ l \ge 5 / 2 $, 
we have 
$
  E( \IW, [ - \ell - 1, \ell + 1 ] ) \ge 0 
$. 
Conversely, if $ \delta_1 \ge 0 $ we get
\begin{equation}
  E( \IW, [0, \ell + 1 ] ) 
  \ge - \frac 32 \mathfrak{N}^{ \IW }( 0 ) - \delta_1
\end{equation}
and for $ \delta_1 \ge \mathfrak{N}^{ \IW }( 0 ) $ and 
$ \ell \ge 9/2 $ we clearly have 
$
  E( \IW, [ - \ell - 1, \ell + 1 ] ) \ge 0 
$. 
For 
$ 
  \delta_1 \le \mathfrak{N}^{ \IW }( 0 ) 
$ 
we get 
\begin{equation}
  E(\IW,[-1,\ell+1]) \ge -\frac 52 \mathfrak N^{\IW}(0) - \frac 12  \delta_1
  .
\end{equation}
Now for 
$ \delta_1 \le \mathfrak{N}^{ \IW }( 0 ) / ( \ell + 1 ) $ 
we get
\begin{equation}
  E(\IW,[-\ell-1,-1]) \ge \frac 12 \ell \mathfrak N^{\IW}(1) 
  \ge \frac 12 (\ell-1) \mathfrak N^{\IW}(0)
\end{equation}
and for $ \ell \ge 7 $ 
we get that $ E( \IW, [-\ell-1,\ell+1]) \ge 0 $. 
If 
$
  \mathfrak N^{\IW}(0)/(\ell+1) \le \delta_1 \le \mathfrak N^{\IW}(0)
$ 
then
\begin{equation}
  E(\IW, [-\ell-1,-1]) \ge \Bigl( \frac \ell 2 \Bigr)^2 |\delta_1| \ge \frac{\ell-1}{4} \mathfrak N^{\IW}(0) 
  .
\end{equation}
Thus, if $ \ell \ge 13 $ have $ E( \IW, [ - \ell - 1, \ell + 1 ] ) \ge 0 $ 
and the proof of the assertion is finished. 
\end{exa}

\subsection*{Acknowledgements}
This work has been partially funded by the Deutsche Forschungsgemeinschaft (DFG, German Research Foundation) 
under Germany's Excellence Strategy EXC 2044--390685587, Mathematics Münster: Dynamics--Geometry--Structure. 
Moreover, this work been partially funded by the Deutsche Forschungsgemeinschaft (DFG, German Research Foundation) 
-- SFB 1283/2 2021 -- 317210226. Furthermore, this work has been partially funded by the European Union (ERC, MONTECARLO, 101045811). 
The views and the opinions expressed in this work are however those of the authors only and do not necessarily reflect 
those of the European Union or the European Research Council (ERC). Neither the European Union nor the granting authority 
can be held responsible for them.

\bibliographystyle{alpha}
\bibliography{Optimal_networks}

\newcommand{\etalchar}[1]{$^{#1}$}
\begin{thebibliography}{EMWW20}

\bibitem[AB09]{attouch2009convergence}
H.~Attouch and J.~Bolte.
\newblock On the convergence of the proximal algorithm for nonsmooth functions
  involving analytic features.
\newblock {\em Math. Program.}, 116(1):5--16, 2009.

\bibitem[AMA05]{MR2197994}
P.-A. Absil, R.~Mahony, and B.~Andrews.
\newblock Convergence of the iterates of descent methods for analytic cost
  functions.
\newblock {\em SIAM J. Optim.}, 16(2):531--547, 2005.

\bibitem[BDL07]{bolte2007lojasiewicz}
J.~Bolte, A.~Daniilidis, and A.~Lewis.
\newblock The \uppercase{{\L}}ojasiewicz inequality for nonsmooth subanalytic
  functions with applications to subgradient dynamical systems.
\newblock {\em SIAM J. Optim.}, 17(4):1205--1223, 2007.

\bibitem[CB18]{ChizatBach2018arXiv}
L.~Chizat and F.~Bach.
\newblock On the global convergence of gradient descent for over-parameterized
  models using optimal transport.
\newblock In {\em Neural Information Processing Systems}, volume~31, 2018.

\bibitem[CB20]{ChizatBach2020arXiv}
L.~Chizat and F.~Bach.
\newblock Implicit bias of gradient descent for wide two-layer neural networks
  trained with the logistic loss.
\newblock In {\em Conference on Learning Theory}, pages 1305--1338. PMLR, 2020.

\bibitem[CJR22]{cheridito2022landscape}
P.~Cheridito, A.~Jentzen, and F.~Rossmannek.
\newblock Landscape analysis for shallow neural networks: complete
  classification of critical points for affine target functions.
\newblock {\em J. Nonlinear Sci.}, 32(5):64, 2022.

\bibitem[CK23]{christof2022omnipresence}
C.~Christof and J.~Kowalczyk.
\newblock On the omnipresence of spurious local minima in certain neural
  network training problems.
\newblock {\em Constr. Approx.}, 2023.

\bibitem[Coo21]{cooper2021global}
Y.~Cooper.
\newblock Global minima of overparameterized neural networks.
\newblock {\em SIAM J. Math. Data Sci.}, 3(2):676--691, 2021.

\bibitem[DDKL20]{MR4056927}
D.~Davis, D.~Drusvyatskiy, S.~Kakade, and J.~D. Lee.
\newblock Stochastic subgradient method converges on tame functions.
\newblock {\em Found. Comput. Math.}, 20(1):119--154, 2020.

\bibitem[DK21]{dereich2021convergence}
S.~Dereich and S.~Kassing.
\newblock Convergence of stochastic gradient descent schemes for
  \uppercase{{\L}}ojasiewicz-landscapes.
\newblock arXiv:2102.09385, 2021.

\bibitem[DK22a]{dereich2021cooling}
S.~Dereich and S.~Kassing.
\newblock Cooling down stochastic differential equations: Almost sure
  convergence.
\newblock {\em Stochastic Process. Appl.}, 152:289--311, 2022.

\bibitem[DK22b]{dereich2022}
S.~Dereich and S.~Kassing.
\newblock On minimal representations of shallow
  \uppercase{R}e\uppercase{L}\uppercase{U} networks.
\newblock {\em Neural Networks}, 148:121--128, 2022.

\bibitem[DK23]{dereich2023central}
S.~Dereich and S.~Kassing.
\newblock Central limit theorems for stochastic gradient descent with averaging
  for stable manifolds.
\newblock {\em Electron. J. Probab.}, 28:1--48, 2023.

\bibitem[DLL{\etalchar{+}}19]{Duetal2019arXivDeepOverparametrized}
S.~S. Du, J.~Lee, H.~Li, L.~Wang, and X.~Zhai.
\newblock Gradient descent finds global minima of deep neural networks.
\newblock In {\em International Conference on Machine Learning}, pages
  1675--1685. PMLR, 2019.

\bibitem[DZPS19]{Duetal2019arXivOverparametrized}
S.~S. Du, X.~Zhai, B.~Poczos, and A.~Singh.
\newblock Gradient descent provably optimizes over-parameterized neural
  networks.
\newblock In {\em International Conference on Learning Representations}, 2019.

\bibitem[EJRW23]{Eberle2021arXiv}
S.~Eberle, A.~Jentzen, A.~Riekert, and G.~S. Weiss.
\newblock Existence, uniqueness, and convergence rates for gradient flows in
  the training of artificial neural networks with {R}e{LU} activation.
\newblock {\em Electron. Res. Arch.}, 31(5):2519--2554, 2023.

\bibitem[EMWW20]{E2020towards}
W.~E, C.~Ma, L.~Wu, and S.~Wojtowytsch.
\newblock Towards a mathematical understanding of neural network-based machine
  learning: What we know and what we don't.
\newblock {\em CSIAM Trans. Appl. Math.}, 1(4):561--615, 2020.

\bibitem[Fou22]{foucart2022mathematical}
S.~Foucart.
\newblock {\em Mathematical pictures at a data science exhibition}.
\newblock Cambridge University Press, 2022.

\bibitem[GJL22]{gallonjentzenlindner2022blowup}
D.~Gallon, A.~Jentzen, and F.~Lindner.
\newblock Blow up phenomena for gradient descent optimization methods in the
  training of artificial neural networks.
\newblock arXiv:2211.15641, 2022.

\bibitem[GK23]{gess2023convergence}
B.~Gess and S.~Kassing.
\newblock Convergence rates for momentum stochastic gradient descent with noise
  of machine learning type.
\newblock arXiv:2302.03550, 2023.

\bibitem[GW22]{GentileWelper2022arXiv}
R.~Gentile and G.~Welper.
\newblock Approximation results for {G}radient {D}escent trained {S}hallow
  {N}eural {N}etworks in 1d.
\newblock arXiv:2209.08399, 2022.

\bibitem[IJR22]{Ibragimov2022arXiv}
S.~Ibragimov, A.~Jentzen, and A.~Riekert.
\newblock Convergence to good non-optimal critical points in the training of
  neural networks: {G}radient descent optimization with one random
  initialization overcomes all bad non-global local minima with high
  probability.
\newblock arXiv:2212.13111, 2022.

\bibitem[JR22]{jentzen2021existence}
A.~Jentzen and A.~Riekert.
\newblock On the existence of global minima and convergence analyses for
  gradient descent methods in the training of deep neural networks.
\newblock {\em Journal of Machine Learning}, 1(2):141--246, 2022.

\bibitem[KKV03]{kainen2003best}
P.~C. Kainen, V.~Kurkov\'{a}, and A.~Vogt.
\newblock Best approximation by linear combinations of characteristic functions
  of half-spaces.
\newblock {\em J. Approx. Theory}, 122(2):151--159, 2003.

\bibitem[Liu21]{liu2021understanding}
B.~Liu.
\newblock Understanding the loss landscape of one-hidden-layer {R}e{L}{U}
  networks.
\newblock {\em Knowledge-Based Systems}, 220:106923, 2021.

\bibitem[LMQ22]{lim2022best}
L.-H. Lim, M.~Micha{\l}ek, and Y.~Qi.
\newblock Best k-layer neural network approximations.
\newblock {\em Constr. Approx.}, 55(1):583--604, 2022.

\bibitem[{\L}oj63]{lojasiewicz1963propriete}
S.~{\L}ojasiewicz.
\newblock Une propri{\'e}t{\'e} topologique des sous-ensembles analytiques
  r{\'e}els.
\newblock {\em Les {\'e}quations aux d{\'e}riv{\'e}es partielles}, 117:87--89,
  1963.

\bibitem[{\L}oj65]{lojasiewicz1965ensembles}
S.~{\L}ojasiewicz.
\newblock Ensembles semi-analytiques.
\newblock {\em Lectures Notes IHES (Bures-sur-Yvette)}, 1965.

\bibitem[{\L}oj84]{lojasiewicz1984trajectoires}
S.~{\L}ojasiewicz.
\newblock Sur les trajectoires du gradient d’une fonction analytique.
\newblock {\em Seminari di geometria}, 1982/1983:115--117, 1984.

\bibitem[LRG24]{le2023does}
Q.-T. Le, E.~Riccietti, and R.~Gribonval.
\newblock Does a sparse \uppercase{R}e\uppercase{L}\uppercase{U} network
  training problem always admit an optimum?
\newblock In {\em Neural Information Processing Systems}, volume~36, 2024.

\bibitem[LXT{\etalchar{+}}18]{li2018visualizing}
H.~Li, Z.~Xu, G.~Taylor, C.~Studer, and T.~Goldstein.
\newblock Visualizing the loss landscape of neural nets.
\newblock In {\em Neural Information Processing Systems}, volume~31, 2018.

\bibitem[MHKC20]{mertikopoulos2020sure}
P.~Mertikopoulos, N.~Hallak, A.~Kavis, and V.~Cevher.
\newblock On the almost sure convergence of stochastic gradient descent in
  non-convex problems.
\newblock In {\em Neural Information Processing Systems}, volume~33, 2020.

\bibitem[PRV21]{MR4243432}
P.~Petersen, M.~Raslan, and F.~Voigtlaender.
\newblock Topological properties of the set of functions generated by neural
  networks of fixed size.
\newblock {\em Found. Comput. Math.}, 21(2):375--444, 2021.

\bibitem[Roc70]{rockafellar1970convex}
R.~T. Rockafellar.
\newblock {\em Convex Analysis}, volume~36.
\newblock Princeton University Press, 1970.

\bibitem[SCP16]{swirszcz2016local}
G.~Swirszcz, W.~M. Czarnecki, and R.~Pascanu.
\newblock Local minima in training of neural networks.
\newblock arXiv:1611.06310, 2016.

\bibitem[Sin70]{singerbest}
I.~Singer.
\newblock {\em Best Approximation in Normed Linear Spaces by Elements of Linear
  Subspaces}.
\newblock Springer, 1970.

\bibitem[SS18]{safran2018spurious}
I.~Safran and O.~Shamir.
\newblock Spurious local minima are common in two-layer
  \uppercase{R}e\uppercase{L}\uppercase{U} neural networks.
\newblock In {\em International Conference on Machine Learning}, pages
  4433--4441. PMLR, 2018.

\bibitem[Tad15]{TADIC2015convergence}
V.~B. Tadić.
\newblock Convergence and convergence rate of stochastic gradient search in the
  case of multiple and non-isolated extrema.
\newblock {\em Stochastic Process. Appl.}, 125(5):1715--1755, 2015.

\bibitem[VBB19]{Venturi2019Spurious}
L.~Venturi, A.~S. Bandeira, and J.~Bruna.
\newblock Spurious valleys in one-hidden-layer neural network optimization
  landscapes.
\newblock {\em J. Mach. Learn. Res.}, 20(133):1--34, 2019.

\bibitem[Woj20]{Wojtowytsch2020arXiv}
S.~Wojtowytsch.
\newblock On the {C}onvergence of {G}radient {D}escent {T}raining for
  {T}wo-layer \uppercase{R}e\uppercase{L}\uppercase{U}-networks in the {M}ean
  {F}ield {R}egime.
\newblock arXiv:2005.13530, 2020.

\bibitem[Woj23]{wojtowytsch2023stochastic}
S.~Wojtowytsch.
\newblock Stochastic gradient descent with noise of machine learning type part
  {I}: {D}iscrete time analysis.
\newblock {\em J. Nonlinear Sci.}, 33(3):45, 2023.

\end{thebibliography}

\end{document}